\documentclass[10pt, article]{amsart}
\usepackage{tikz}
\usetikzlibrary{calc}
\usepackage{ae} 
\usepackage[T1]{fontenc}
\usepackage[cp1250]{inputenc}
\usepackage{amsmath}
\usepackage{amssymb, amsfonts,amscd,verbatim}

\usepackage[normalem]{ulem}
\usepackage{hyperref}
\usepackage{indentfirst}
\usepackage{latexsym}
\input xy
\xyoption{all}

\usepackage{amsmath}    

\theoremstyle{plain}
\newtheorem{Pocz}{Poczatek}[section]
\newtheorem{Proposition}[Pocz]{Proposition}
\newtheorem{Conjecture}[Pocz]{Conjecture}
\newtheorem{Theorem}[Pocz]{Theorem}
\newtheorem{Corollary}[Pocz]{Corollary}

\newtheorem{Lemma}[Pocz]{Lemma}
\newtheorem{Observation}[Pocz]{Observation}

\newtheorem{Question}[Pocz]{Question}
\newtheorem{Problem}[Pocz]{Problem}
\newtheorem{Example}[Pocz]{Example}

\theoremstyle{definition}
\newtheorem{Definition}[Pocz]{Definition}

\theoremstyle{remark}
\newtheorem{Remark}[Pocz]{Remark}

\DeclareMathOperator*{\diam}{diam}

\numberwithin{equation}{section}
\title[Linear algebra and unification of geometries in all scales]
{Linear algebra and unification of geometries in all scales}

\author{J.~Dydak}
\address{University of Tennessee, Knoxville, TN 37996, USA}
\email{jdydak@utk.edu}

\date{ \today
}


\begin{document}
\maketitle
\begin{center}
\today
\end{center}

\tableofcontents

\begin{abstract}
We present an idea of unifying small scale (topology, proximity spaces, uniform spaces) and large scale (coarse spaces, large scale spaces). It relies on an analog of multilinear forms from Linear Algebra. Each form has a large scale compactification and those include all well-known compactifications:
Higson corona, Gromov boundary of hyperbolic spaces, the visual boundary of CAT(0)-spaces, \v Cech-Stone compactification, Samuel-Smirnov compactification, and Freudenthal compactification. 

As an application we get simple proofs of results generalizing well-known theorems from coarse topology. A new result (at least to the author) is the following (see \ref{HomeoOfHigsonImpliesLSEquivalence}):\\
\emph{A coarse bornologous function $f:X\to Y$ of metrizable large scale spaces is a large scale equivalence if and only if it induces a homeomorphism of Higson coronas.}

This paper is an extension of \cite{JD2} and, at the same time, it overrides \cite{JD2}. 
\end{abstract}

\section{Introduction}

A topology on a set $X$ is the same as a projection (i.e. an idempotent linear operator)
$cl:2^X\to 2^X$ satisfying $A\subset cl(A)$ for all $A\subset X$. That's a good way to summarize Kuratowski's closure operator.

Basic geometry on a set $X$ is a dot product $\cdot:2^X\times 2^X\to \{0,\infty\}$. Its equivalent form is an orthogonality relation on subsets of $X$. The optimal case is when the orthogonality relation satisfies a variant of parallel-perpendicular decomposition from linear algebra. Dot products are a special case of forms which act on arbitrary vectors based on a given set $X$.

We show that this concept unifies small scale (topology, proximity spaces, uniform spaces) and large scale (coarse spaces, large scale spaces). Using forms we define large scale compactifications that generalize all well-known compactifications:
Higson corona, Gromov boundary, the visual boundary of CAT(0)-spaces, \v Cech-Stone compactification, Samuel-Smirnov compactification, and Freudenthal compactification. This allows to generalize many results in coarse topology from proper metric spaces to arbitrary metric spaces or even to arbitrary large scale spaces.

\begin{Example} (see \ref{GActingOnMetricSpaces})
Let $X$ be a metric space and let $G$ be a finite group acting isometrically on X. Then $X/G$ has the same asymptotic dimension as $X$.

In case of proper metric spaces $X$, Theorem \ref{GActingOnMetricSpaces}
was proved by Daniel Kasprowski \cite{Kasp}.

\end{Example}

\begin{Example} (see \ref{GeneralCoarselyNTo1Thm})
If $n\ge 1$ and $f:X\to Y$ is a coarsely $n$-to-$1$ bornologous map of large scale spaces, then $asdim(Y)\leq asdim(Y)+n-1$.

Theorem \ref{GeneralCoarselyNTo1Thm} was proved by Austin-Virk in \cite{AV} for proper metric spaces $X$ and $Y$.

\end{Example}

\section{Multilinear forms on sets}

\begin{Definition}
The semi-group $\{0,\infty\}$ has the following binary operation:\\
1. $0+0=0$,\\
2. $0+\infty=\infty+0=\infty+\infty=\infty$.
\end{Definition}

Recall that a \textbf{bornology} $\mathcal{B}$ on a set $X$ is any family of subsets closed under finite unions 
so that $B\subset B'\in\mathcal{B}$ implies $B\in\mathcal{B}$.

\begin{Remark}
Be aware that Wikipedia \cite{WikiBornology} assumes additionally that each point of $X$ belongs to a bornology.
\end{Remark}

Notice that bornologies $\mathcal{B}$ on a set $X$ are identical with kernels of basic linear operators $\omega:2^X\to \{0,\infty\}$, i.e. functions satisfying\\
1. $\omega(\emptyset)=0$,\\
2. $\omega(C\cup D)=\omega(C)+\omega(D)$ for all $C,D\in 2^X$.

That observation leads to the following generalization:
\begin{Definition}
A \textbf{$k$-vector} in $X$ ($k\ge 1$) is a $k$-tuple $(C_1,\ldots,C_k)$ of subsets in $X$. The set of all vectors in $X$ will be denoted by $Vect(X)$.

The \textbf{concatenation} $V_1\ast V_2$ of a $k$-vector $V_1=(C_1,\ldots,C_k)$
and an $m$-vector $V_2=(D_1,\ldots,D_m)$ is the $(k+m)$-vector
$(C_1,\ldots,C_k,D_1,\ldots,D_m)$.

A \textbf{basic multilinear form} $\omega$ on $X$ is a symmetric function on all $k$-vectors of $X$ ($k\ge 1$) to $\{0,\infty\}$ satisfying the following properties:\\
1. $\omega((C_1\cup C_2)\ast V)= \omega(C_1\ast V)+ \omega(C_2\ast V)$
for any $k$-vector $V$ and any two $1$-vectors $C_1$, $C_2$,\\
2. $\omega(\emptyset)=0$,\\
3. $\omega(C\ast V)=\omega(V)$ if one of the coordinates of $V$ is contained in $C$.
\end{Definition}

\begin{Remark}
One can consider forms with values in any \textbf{join-semilattice}, i.e. a partially ordered set $(L,\leq)$ in which each two-element subset $\{a, b\}\subset L$ has a join (i.e. least upper bound). However, it is not clear what new applications would arise from that concept.
\end{Remark}

Notice that $\{0,\infty\}$ has a natural order: namely $0 < \infty$ (in fact, it is the smallest non-trivial lattice).
\begin{Lemma}\label{OmegaInequalityOne}
If $V$ and $V'$ are two $k$-vectors such that $V(i)\subset V'(i)$ for each $1\leq i\leq k$, then $\omega(V)\leq \omega(V')$.
\end{Lemma}
\begin{proof}
Apply Axiom 1 repeatedly to see that $\omega(V)$ is a summand of $\omega(V')$.
\end{proof}

\begin{Corollary}\label{OmegaInequalityTwo}
$\omega(V)\leq \omega(C)$ if some coordinate of $V$ is contained in $C$.
\end{Corollary}
\begin{proof}
If $V$ is a $k$ vector, we may assume $V(k)\subset C$. Now,
$\omega(V)\leq \omega(X\ast\ldots \ast X\ast C)$ by 
\ref{OmegaInequalityOne} and the latter equals $\omega(C)$ by Axiom 3.
\end{proof}

\begin{Proposition}\label{OmegaInequalityThree}
$\omega(V\ast V')\leq \omega(V)$ for all vectors $V$ and $V'$.
\end{Proposition}
\begin{proof}
By \ref{OmegaInequalityOne}, $\omega(V\ast V')\leq \omega(V\ast X\ast\ldots \ast X)$ and the latter equals $\omega(V)$ by Axiom 3.
\end{proof}

\begin{Definition}\label{FormedSpaceDefinition}
A \textbf{bornological set} $(X,\mathcal{B})$ is a set $X$ equipped with a bornology $\mathcal{B}$.

A \textbf{formed set} $(X,\omega)$ is a set $X$ equipped with a basic multilinear form $\omega$.
\end{Definition}

\begin{Proposition}
Each formed set $(X,\omega)$ induces a bornology on $X$ via
$$\mathcal{B}(\omega):=\{C\subset X\mid \omega(C)=0\}.$$

Conversely, given a bornology $\mathcal{B}$ on $X$,
$$\omega(C_1,\ldots,C_k)=0\iff C_1\cap\ldots \cap C_k\in \mathcal{B}$$
defines a basic multilinear form on $X$ that induces $\mathcal{B}$.
\end{Proposition}
\begin{proof}
Left to the reader.
\end{proof}

\begin{Observation}
Notice that, given a bornology $\mathcal{B}$ on $X$, the form defined above is the largest form that induces $\mathcal{B}$. The smallest form inducing $\mathcal{B}$ is given by the formula
$$\omega(C_1,\ldots,C_k)=0\iff C_i\in \mathcal{B} \mbox{ for some } i.$$
\end{Observation}

\begin{Example}
Given a topological space $(X,\mathcal{T})$, the \textbf{basic topological form} $\omega(X,\mathcal{T})$ is defined as follows:\\
$$\omega(X,\mathcal{T})(C_1,\ldots,C_k)=0 \iff
\bigcap\limits_{i=1}^k cl(C_i)=\emptyset.$$
\end{Example}

\begin{Example}
Given a topological space $(X,\mathcal{T})$, the \textbf{basic functional form} $\omega_f(X,\mathcal{T})$ is defined as follows:\\
$$\omega_f(X,\mathcal{T})(C_1,\ldots,C_k)=0 \iff
\bigcap\limits_{i=1}^k D_i=\emptyset \mbox{ for some zero-sets } D_i\supset C_i.$$

Recall $D$ is a \textbf{zero-set} in a topological space $X$ if there is a continuous function $f:X\to [0,1]$ such that $D=f^{-1}(0)$.
\end{Example}

\begin{Example}
Given a metric space $(X,d)$, the \textbf{basic small scale form} $\omega_s(X,d)$ is defined as follows:\\
$$\omega_s(X,d)(C_1,\ldots,C_k)=0 \iff
\bigcap\limits_{i=1}^k B(C_i,r)=\emptyset$$
for some $r > 0$.
\end{Example}

\begin{Example}\label{BasicLargeScaleFormOnMetric}
Given a metric space $(X,d)$, the \textbf{basic large scale form} $\omega_l(X,d)$ is defined as follows:\\
$$\omega_l(X,d)(C_1,\ldots,C_k)=0 \iff
\bigcap\limits_{i=1}^k B(C_i,r)\mbox{ is bounded}$$
for each $r > 0$.
\end{Example}

\begin{Example}
Given a metric space $(X,d)$, the $C_0$-form $\omega_0(X,d)$ is defined as follows:\\
$$\omega_0(X,d)(C_1,\ldots,C_k)=0 \iff
\bigcap\limits_{i=1}^k B(C_i,r)\mbox{ is bounded}$$
for some $r > 0$.
\end{Example}

Recall that the \textbf{star} $st(B,\mathcal{U} )$ of a subset $B$ of $X$ with respect
to a family $\mathcal{U}$ of subsets of $X$ is the union of those elements of $\mathcal{U}$
that intersect $B$.
Given two families $\mathcal{B}$ and $\mathcal{U}$ of subsets of $X$,
$st(\mathcal{B},\mathcal{U})$ is the family $\{st(B,\mathcal{U})\}$, $B\in\mathcal{B}$,
of all stars of elements of $\mathcal{B}$ with respect to $\mathcal{U}$.

\begin{Example}\label{FormOnUniformSpaces}
Given a uniform space $(X,\mathcal{U})$, the form $\omega(\mathcal{U})$ is defined as follows:\\
$$\omega(\mathcal{U})(C_1,\ldots,C_k)=0 \iff
\bigcap\limits_{i=1}^k st(C_i,\mathcal{V})=\emptyset$$
for some uniform cover $\mathcal{V}$ of $X$.
\end{Example}

\begin{Example}
Given a topological space $(X,\mathcal{T})$ and a subset $A$ of $X$, the form $\omega(X,A,\mathcal{T})$ on $X\setminus A$ is defined as follows:\\
$$\omega(X,A,\mathcal{T})(C_1,\ldots,C_k)=0 \iff
\bigcap\limits_{i=1}^k cl(C_i)\subset X\setminus A.$$
Here the closures are taken in $X$ via $\mathcal{T}$.
\end{Example}

\section{Topology induced by forms}

\begin{Proposition}
Given a formed set $(X,\omega)$, the \textbf{topology induced} by $\omega$ on $X$ consists of sets $U$ with the property that $\omega(x,X\setminus U)=0$ for all $x\in U$.
\end{Proposition}

\begin{proof}
If $x\in U\cap W$ and $\omega(x,X\setminus U)=0=\omega(x,X\setminus W)=0$, then $\omega(x,(X\setminus U)\cup (X\setminus W))=0$
and $(X\setminus U)\cup (X\setminus W)=X\setminus U\cap W$.

If each $U_s$, $s\in S$, is open and $x\in W:=\bigcup\limits_{s\in S}U_s$,
then $x\in U_t$ for some $t\in S$, so $\omega(x,X\setminus U_t)=0$.
As $X\setminus W\subset X\setminus U_t$, $\omega(x,X\setminus W)=0$
by \ref{OmegaInequalityOne}.
\end{proof}

Here is a typical construction of open sets in the topology induced by $\omega$:
\begin{Proposition}\label{ConstructionOfOpenSetsViaForms}
If $\{C_n\}_{n=1}^\infty$ is an increasing sequence of subsets of a formed set $(X,\omega)$ such that $\omega(C_k,X\setminus C_{k+1})=0$ for each $k\ge 1$, then $U:=\bigcup\limits_{i=1}^\infty C_k$ is an open set in the topology induced by $\omega$.
\end{Proposition}
\begin{proof}
Given $x\in U$, there is $k\ge 1$ such that $x\in C_k$.
Therefore $\omega(x,X\setminus C_{k+1})=0$. Since $X\setminus U\subset X\setminus C_{k+1}$, $\omega(x,X\setminus U)=0$ by \ref{OmegaInequalityOne}.

\end{proof}

\begin{Proposition}
The topology induced by $\omega$ is $T_1$ if and only if $\omega(x,y)=0$ whenever $x\ne y$.
\end{Proposition}
\begin{proof}
If $\omega(x,y)=0$ whenever $x\ne y$, then
$X\setminus \{y\}$ is open.
If $X\setminus \{y\}$ is open and $x\ne y$, then $\omega(x,y)=0$.
\end{proof}

\begin{Definition}
A form $\omega$ on a set $X$ is $T_1$ if $\omega(x,y)=0$ for all different points $x,y\in X$.
\end{Definition}

Forms cannot induce topologies that are $T_0$ but not $T_1$:
\begin{Corollary}
If the topology $\mathcal{T}$ induced by a form $\omega$ on $X$ is $T_0$, then $(X,\mathcal{T})$ is $T_1$.
\end{Corollary}
\begin{proof}
Given $x\ne y$ in $X$, there is an open set $U$ containing exactly one of the points $x$, $y$. Say $y\in U$ and $x\notin U$. Therefore $\omega(y,X\setminus U)=0$ implying $\omega(y,x)=0$.
\end{proof}

\begin{Proposition}
If $(X,\mathcal{T})$ is a $T_1$ topological space, then its basic topological form induces $\mathcal{T}$.
\end{Proposition}
\begin{proof}
Let $\omega$ be the basic topological form of $X$.
$\omega(x,X\setminus U)=0$ means $cl(x)\cap cl(X\setminus U)=\emptyset$.
Consequently, if $U\in \mathcal{T}\in T_1$, then $U$ is open in the topology induced by $\omega$. Conversely, if $C$ is closed in the topology induced by $\omega$ and $x\notin C$, then $\omega(x,C)=0$, hence $x\notin cl(C)$.
Thus $C=cl(C)$ in $\mathcal{T}$.
\end{proof}

Of special interest are cases where the topology induced by a form has properties related to compactness.

\subsection{Large scale topology}
\begin{Definition}
A \textbf{large scale topological space} $(X,\mathcal{T},\mathcal{B})$ is a topological space $(X,\mathcal{T})$ in which a bornology $\mathcal{B}$
of open-closed subspaces is selected.
\end{Definition}

\begin{Definition}
A large scale topological space $(X,\mathcal{T},\mathcal{B})$ is \textbf{large scale compact} if and only if, for any family $\{U_s\}_{s\in S}$ of open subsets of $X$, $X=\bigcup\limits_{s\in S}U_s$ implies existence of a finite subset $F$ of $S$ such that
$X\setminus \bigcup\limits_{s\in F}U_s$ belongs to $\mathcal{B}$.
\end{Definition}

\begin{Proposition}\label{NormalityOfLargeScaleCompactSpaces}
1) If $(X,\mathcal{T},\mathcal{B})$ is large scale compact and topologically Hausdorff, then it is topologically regular.\\
2)  If $(X,\mathcal{T},\mathcal{B})$ is large scale compact and topologically regular, then it is topologically normal.
\end{Proposition}
\begin{proof}
1). Suppose $A$ is a closed subset of $X$ not containing $x_0$. If $x_0$ is open-closed, then $A\subset X\setminus \{x_0\}$ is disjoint from $x_0$ and we are done. Assume $x_0$ is not open.
For each point $x\in A$ choose disjoint open sets $U_x$ containing $x$ and $V_x$ containing $x_0$.
Notice $X=(X\setminus A)\cup\bigcup\limits_{x\in A}U_x$, so there is an open-closed set $B\in \mathcal{B}$ and a finite subset $F$ of $A$ such that
$X=B\cup (X\setminus A)\cup \bigcup\limits_{x\in A}U_x$ and $B$ does not contain $x_0$. Therefore, $A\subset B\cup \bigcup\limits_{x\in A}U_x$ is disjoint from
$\bigcap\limits_{x\in F}V_x\setminus B$ which contains $x_0$.\\
The proof of 2) is similar or apply 1) to $X/A$.
\end{proof}

\begin{Proposition}\label{ClosedMapProp}
A continuous map $f:(X,\mathcal{T}_X,\mathcal{B}_X)\to (Y,\mathcal{T}_Y,\mathcal{B}_Y)$ of large scale topological spaces is closed if the following conditions are satisfied:\\
1. $(X,\mathcal{T}_X,\mathcal{B}_X)$ is large scale compact.\\
2. $(Y,\mathcal{T}_Y,\mathcal{B}_Y)$ is Hausdorff.\\
3. $f(\mathcal{B}_X)\subset \mathcal{B}_Y$.
\end{Proposition}
\begin{proof}
It suffices to show $f(X)$ is closed and then apply it to $f|A$, $A$ any closed subset of $X$.\\
Given $y\in Y\setminus f(X)$ that is not open choose, for any $z\in f(X)$, disjoint open neighborhoods $U(z)$ of $y$ and $W(z)$ of $z$.
Since $\{f^{-1}(W(z))\}_{z\in f(X)}$ covers $X$, there exist points $z_1,\ldots,z_k$ of $f(X)$ such that $X\setminus\bigcup\limits_{i=1}^k f^{-1}(W(z_i))\in \mathcal{B}_X$. Hence, $D:=f(X)\setminus\bigcup\limits_{i=1}^k W(z_i)\in \mathcal{B}_Y$ is open-closed in $Y$, it does not contain $y$,
and $U:=\bigcap\limits_{i=1}^k U(z_i)\setminus D$ is a neighborhood of $y$ missing $f(X)$.
\end{proof}

\section{Boundaries of large scale topological spaces}

In this section we generalize \v Cech-Stone compactification to large scale compactifications of large scale topological spaces.
\begin{Definition}
The \textbf{boundary} $\partial (X,\mathcal{T},\mathcal{B})$
of a large scale topological space $(X,\mathcal{T},\mathcal{B})$ is the collection of all maximal families
$\mathcal{P}$ in $2^X$ consisting of zero-sets satisfying the following condition:
intersection of any finite subfamily of $\mathcal{P}$ is not in ${\mathcal{B}}$.

If $\{x\}\notin {\mathcal{B}}$, then we can identify $x$ with the family consisting of all zero subsets of $X$ containing $x$. This way we can talk about
$X\cup \partial (X,\mathcal{T},\mathcal{B})$.
\end{Definition}

\begin{Observation}
Technically speaking, points $x_1, x_2\in X$ such that there is no continuous function from $X$ to $[0,1]$ separating them, are identified in $\partial (X,\mathcal{T},\mathcal{B})$. Namely, they generate the same element of $\partial (X,\mathcal{T},\mathcal{B})$ consisting of all zero subsets of $X$ containing $x_1$. Thus, to consider
$X\cup \partial (X,\mathcal{T},\mathcal{B})$ in set-theoretic sense, we need $X$ to be functionally Hausdorff (i.e. every pair of different points of $X$ can be separated by a real-valued continuous function). We will continue to use notation $X\cup \partial (X,\mathcal{T},\mathcal{B})$ for the union of a quotient of $X$ and $\partial (X,\mathcal{T},\mathcal{B})$
\end{Observation}

\begin{Remark}
In case of a bornological set $(X,\mathcal{B})$ we can define its boundary as $\partial(X,\mathcal{T},\mathcal{B})$, where $\mathcal{T}$ is the discrete topology on $X$.
\end{Remark}

\begin{Observation}\label{ObservationOnPointsAtBdOfBornology}
Suppose $A_1$ and $A_2$ are zero-subsets of $(X,\mathcal{T},\mathcal{B})$.
If $A_1\cup A_2\in \mathcal{P}\in \partial (X,\mathcal{T},\mathcal{B})$, then $A_1\in \mathcal{P}$ or $A_2\in \mathcal{P}$.
\end{Observation}
\begin{proof}
Suppose $A_1\cup A_2\in \mathcal{P}$ but $A_1\notin \mathcal{P}$ and $A_2\notin \mathcal{P}$. In that case there exist $C, D\in \mathcal{P}$
such that $A_1\cap C, A_2\cap D\in \mathcal{B}$. Consequently,
$(A_1\cup A_2)\cap C\cap D\in \mathcal{B}$, a contradiction.
\end{proof}

\begin{Corollary}\label{SetMinusBorno}
If $A\in \mathcal{P}$ and $B\in \mathcal{B}$, then $A\setminus B\in P$.
\end{Corollary}

\begin{Definition}
$X\cup \partial (X,\mathcal{T},\mathcal{B})$ has a natural topology whose basis 
is formed (in view of \ref{NUIsaBasis}) by sets 
$$N(U):=U\cup\{\mathcal{Q}\in \partial(X,\mathcal{T},\mathcal{B})\mid X\setminus U\notin \mathcal{Q}\},$$
where $U$ is a co-zero-subset of $X$ (i.e. the complement of a zero-set in $X$).
\end{Definition}

\begin{Lemma}\label{NUIsaBasis}
$N(U)\cap N(V)=N(U\cap V)$ and $N(U)\cup N(V)=N(U\cup V)$.
\end{Lemma}
\begin{proof}
Reformulate \ref{ObservationOnPointsAtBdOfBornology} as follows:
$A_1\cup A_2\notin \mathcal{P}\in \partial(X,\mathcal{B})$ if and only if $A_1\notin \mathcal{P}$ and $A_2\notin \mathcal{P}$.
The equality $N(U)\cap N(V)=N(U\cap V)$ is now obvious as well as 
$N(U)\cup N(V)= N(U\cup V)$.
\end{proof}

\begin{Corollary}\label{BisOpenClosed}
Each $B\in \mathcal{B}$ is open-closed in $X\cup \partial (X,\mathcal{T},\mathcal{B})$.
\end{Corollary}
\begin{proof}
By \ref{SetMinusBorno}, $N(B)=B$ and $N(X\setminus B)=X\setminus B$. As $N(X)=N(B)\cup N(X\setminus B)$, both $N(B)$ and $N(X\setminus B)$ are closed.
\end{proof}

\begin{Proposition}
Suppose $C$ is a zero-subset of $X$.
$\mathcal{P}\in \partial(X,\mathcal{T},\mathcal{B})$ belongs to the closure of $C\subset X$ if and only if $C\in \mathcal{P}$.
\end{Proposition}
\begin{proof} 
If $C\in \mathcal{P}$, then every base element $N(U)$ containing $P$
intersects $C$ as $C\cap U$ cannot be empty. Indeed, $C\cap U=\emptyset$ means $C\subset X\setminus U$ resulting in $X\setminus U\in \mathcal{P}$ which contradicts $\mathcal{P}\in N(U)$.

Conversely, if $C\notin \mathcal{P}$, then $\mathcal{P}\in N(X\setminus C)$ and
$C\cap N(X\setminus C)=\emptyset$. Indeed, any $\mathcal{Q}\in C\cap N(X\setminus C)$ must be equal to all zero subsets containing some $x\in C$. Hence $C\in Q$ which contradicts $\mathcal{Q}\in N(X\setminus C)$.
\end{proof}

\begin{Corollary}
The basis of $X\cup \partial (X,\mathcal{B})$ consists of open-closed sets.
\end{Corollary}
\begin{proof}
In this case the topology on $X$ is discrete and the complement of $N(U)$ is $N(X\setminus U)$.
\end{proof}

\begin{Theorem}\label{XCupBoundaryIsCompactHausdorff}
$X\cup \partial (X,\mathcal{T},\mathcal{B})$ is large scale compact and topologically normal.
\end{Theorem}
\begin{proof}
Suppose $\{N(U)\}_{U\in \mathcal{S}}$ is a cover of $X\cup \partial (X,\mathcal{T},\mathcal{B})$
such that $X\cup \partial (X,\mathcal{T},\mathcal{B})\setminus \bigcup\limits_{U\in \mathcal{F}}N(U)$ does not belong to $\mathcal{B}$ for all finite subfamilies $\mathcal{F}$ of $\mathcal{S}$. Consider the family
$\{X\setminus \bigcup\limits_{U\in \mathcal{F}}U\}_{\mathcal{F}\subset \mathcal{S}}$, where $\mathcal{F}$ runs over all finite subfamilies of $\mathcal{S}$. That family has the property that it is closed under finite intersections and none of its elements belongs to $\mathcal{B}$. Therefore it extends to an element $\mathcal{P}$ of $\partial (X,\mathcal{T},\mathcal{B})$.
Hence, $\mathcal{P}\in N(U)$ for some $U\in \mathcal{S}$. However, $X\setminus U$ also belongs to $\mathcal{P}$, a contradiction.

To prove $X\cup \partial (X,\mathcal{T},\mathcal{B})$ is topologically normal it suffices to show it is Hausdorff. Clearly, any two points $x_1\ne x_2$ in $X$ have disjoint neighborhoods in $X\cup \partial (X,\mathcal{T},\mathcal{B})$ if $x_1\in \mathcal{B}$. Namely, it is $\{x_1\}$ and its complement.
If $\mathcal{P}\ne \mathcal{Q}$ are points of $\partial (X,\mathcal{T},\mathcal{B})$, then there is $C\in \mathcal{P}\setminus \mathcal{Q}$. Hence, there is $D\in \mathcal{Q}$ such that $C\cap D\in \mathcal{B}$. Notice $E:=D\setminus C\cap D\in \mathcal{Q}$
(see \ref{SetMinusBorno})
and there exist disjoint co-zero subsets  $U$ and $W$ of $X$ containing $C$ and $E$, respectively. Observe $\mathcal{P}\in N(U)$,
$\mathcal{Q}\in N(W)$, and $N(U)\cap N(W)=\emptyset$.
\end{proof}

\begin{Theorem}\label{FundThmOfLSCompactifications}
If $f:(X,\mathcal{T}_X,\mathcal{B}_X)\to (Y,\mathcal{T}_Y,\mathcal{B}_Y)$ is a continuous function of
large scale topological spaces such that $f^{-1}(\mathcal{B}_Y)\subset \mathcal{B}_X$, then there is a unique continuous extension
$$\tilde f:X\cup \partial (X,\mathcal{T}_X,\mathcal{B}_X)\to Y\cup \partial(Y,\mathcal{T}_Y,\mathcal{B}_Y)$$
of $f$. It is 
given by
$$\tilde f(\mathcal{P})\in \bigcap\limits_{C\in \mathcal{P}} cl(f(C))$$
for all $\mathcal{P}\in \partial (X,\mathcal{T}_X,\mathcal{B}_X)$.
\end{Theorem}
\begin{proof}
First of all, let us show that $\bigcap\limits_{C\in \mathcal{P}} cl(f(C))\ne\emptyset$ for all $\mathcal{P}$, i.e. $\tilde f$ exists. Suppose $\bigcap\limits_{C\in \mathcal{P}} cl(f(C))=\emptyset$
for some $\mathcal{P}$. Since $Y\cup \partial(Y,\mathcal{T}_Y,\mathcal{B}_Y)$ is large scale compact, there exist $C_i\in P$, $1\leq i\leq k$, such that $\bigcap\limits_{i=1}^k cl(f(C_i))\in \mathcal{B}_Y$. Let $C:=\bigcap\limits_{i=1}^k C_i$. Now,
$C\in \mathcal{P}$ and $f(C)\in \mathcal{B}_Y$.
Therefore $C\in \mathcal{B}_X$, a contradiction.

Our next step is to show continuity of $\tilde f$.
Suppose $\tilde f(\mathcal{P})\in N(U)$ for some co-zero set $U$ in $Y$.
By a similar argument as above, $D\subset U$ for some
$D\in \tilde f(\mathcal{P})$, so we can choose a co-zero set $W$ in $Y$ and a zero-set $E$ in $Y$ satisfying $D\subset W\subset E\subset U$. Suppose $\mathcal{Q}\in N(f^{-1}(W))$.
There is $C\in \mathcal{Q}$, $C\subset f^{-1}(W)$. Hence
$\tilde f(\mathcal{Q})\in cl(f(C))\subset cl(W)\subset cl(E)$ which results
in $E\in \tilde f(\mathcal{Q})$. Now, $Y\setminus U\notin \tilde f(\mathcal{Q})$,
so $\tilde f(\mathcal{Q})\in N(U)$ which completes the proof of continuity of $\tilde f$ provided $\mathcal{P}\in N(f^{-1}(W))$.
It is so as $G:=X\setminus f^{-1}(W)\in \mathcal{P}$ implies
$\tilde f(\mathcal{P})\in cl(f(G))\subset cl(Y\setminus W)$,
i.e. $Y\setminus W\in \tilde f(\mathcal{P})$ which contradicts 
$D\in \tilde f(\mathcal{P})$.

The uniqueness of $\tilde f$ follows from the fact the range is Hausdorff (see \ref{XCupBoundaryIsCompactHausdorff}) and
$X$ is dense in $X\cup \partial (X,\mathcal{T}_X,\mathcal{B}_X)$.
\end{proof}

\begin{Observation}
\ref{FundThmOfLSCompactifications} is a generalization of the classical \v Cech-Stone compactification. Namely, it reduces to the basic property of \v Cech-Stone compactification in the case of bornologies being empty. See \ref{BoundaryOfLSTopInCompactCase}
for more details.
\end{Observation}

\begin{Proposition}\label{BoundaryOfLSTopInCompactCase}
If $(X,\mathcal{T}_X,\mathcal{B}_X)$ is large scale compact, then for every $\mathcal{P}\in \partial(X,\mathcal{T}_X,\mathcal{B}_X)$
there is $x\in X\setminus \bigcup \mathcal{B}_X$ such that
$P$ equals the family of all zero-sets in $X$ containing $x$.
In particular, if $(X,\mathcal{T}_X,\mathcal{B}_X)$ is large scale compact and Hausdorff, then $\partial(X,\mathcal{T}_X,\mathcal{B}_X)=X\setminus \bigcup \mathcal{B}_X$ and
$id:(X,\mathcal{T}_X,\mathcal{B}_X)\to X\cup \partial(X,\mathcal{T}_X,\mathcal{B}_X)$ is a homeomorphism.
\end{Proposition}
\begin{proof}
Suppose for each $x\in X\setminus \bigcup \mathcal{B}_X$ there is
$C_x\in \mathcal{P}$ not containing $x$. For $x\in \bigcup \mathcal{B}_X$ put $C_x:=X\setminus \{x\}$.
Since $X=\bigcup\limits_{x\in X}(X\setminus C_x)$,
there exists a finite subset $F$ of $X$ such that $\bigcap\limits_{x\in F}C_x\in \mathcal{B}_X$. However, $ \bigcap\limits_{x\in F}C_x\in \mathcal{P}$, a contradiction.
\end{proof}

\section{Boundaries of formed sets}
In this section we generalize boundaries of large scale topological spaces to boundaries of formed sets.
\begin{Definition}
Given a set $X$ equipped with a basic multilinear form $\omega$,
the \textbf{boundary} $\partial(\omega)$ is defined as the set of maximal families $\mathcal{P}$ of subsets of $X$ so that
$$\omega(C_1,\ldots,C_k)=\infty$$
for all sets $C_1,\ldots,C_k\in  \mathcal{P}$. Such $\mathcal{P}$ may be referred to as a \textbf{point at infinity} of $X$ despite the possibility of $\mathcal{P}$ being identified with some point of $X$.
\end{Definition}

\begin{Definition}
If $\{x\}$ is $\omega$-unbounded, then we can identify $x$ with the principal ultrafilter consisting of all subsets of $X$ containing $x$. This way we can talk about
$X\cup \partial (\omega)$. To distinguish between $X$ and its boundary, for each $C\subset X$ we define its \textbf{non-boundary points} $C^o$ as
$\{x\in C | \omega(x)=0\}=C\setminus \partial(\omega)$.
\end{Definition}

\begin{Remark}
There may be points $\mathcal{P}$ of the boundary containing $x\in X$ but not equal to $x$. That is the case if $\omega(x)=\omega(X\setminus\{x\})=\omega(x,X\setminus\{x\})=\infty$, for example.
\end{Remark}

\begin{Observation}\label{ObservationOnPointsAtBdOfForm}
If $C\cup D\in \mathcal{Q}\in \partial(\omega)$, then $C\in \mathcal{Q}$ or $D\in \mathcal{Q}$.
\end{Observation}
\begin{proof}
Suppose $C\cup D\in \mathcal{Q}$ but $C\notin \mathcal{Q}$ and $D\notin \mathcal{Q}$. In that case there exist $V_1, V_2\in Vect(\mathcal{Q})$
such that $\omega(C\ast V_1)=\omega(D\ast V_2)=0$. Consequently,
$\omega((C\cup D)\ast V_1\ast V_2)=0$ by \ref{OmegaInequalityThree}, a contradiction.
\end{proof}

In order to extend the topology on $(X,\omega)$ induced by $\omega$ over $X\cup \partial (\omega)$ we need the following concept:
\begin{Definition}
Suppose $(X,\omega)$ is a formed set and $C\subset X$.
$o(C)$ is defined as
$$o(C):=C^o\cup \{\mathcal{Q}\in \partial (\omega)\mid X\setminus C\notin \mathcal{Q}\},$$
\end{Definition}

\begin{Observation}
Notice $C=o(C)\cap X$. Indeed, if $x\in C\setminus C^o$, then
$X\setminus C$ does not belong to the principal ultrafilter generated by $x$.
Conversely, if $x\in X\cap o(C)$, then either $x\in C^0\subset C$
or $x$ represents the principal ultrafilter not containing $X\setminus C$
which implies $x\in C$.
\end{Observation}

\begin{Lemma}\label{IntersectionOfos}
$o(C)\cap o(D)=o(C\cap D)$.
\end{Lemma}
\begin{proof}
The equality follows from \ref{ObservationOnPointsAtBdOfForm}
and $(C\cap D)^o=C^o\cap D^o$.
\end{proof}

\begin{Definition}
Suppose $(X,\omega)$ is a set $X$ equipped with a basic multilinear form $\omega$ and $\mathcal{B}(\omega)$ is the bornology induced by $\omega$.
$X\cup \partial (\omega)$ has a \textbf{natural topology} defined as follows:
$U\subset X\cup \partial (\omega)$ is declared open if $X\cap U$ is open in $\mathcal{T}(\omega)$ and for each $Q\in U\cap \partial (\omega)$ there is an open (in $\mathcal{T}(\omega)$)
$C\subset X$ such that $Q\in o(C)\subset U$.
\end{Definition}

\begin{Remark}
In view of \ref{IntersectionOfos}, the above topology is well-defined and sets $o(U)$, $U$ open in $X$, form its basis.
\end{Remark}

\begin{Lemma}\label{ClosureCLemma}
1. $\mathcal{Q}\in \partial(\omega)$ belongs to $cl(C)$, $C\subset X$, if
$C\in \mathcal{Q}$.\\
2. If $C$ is closed in $X$ and $\mathcal{Q}\in cl(C)$, then
$C\in \mathcal{Q}$.
\end{Lemma}
\begin{proof} 
1. Suppose $C\in \mathcal{Q}$ and $\mathcal{Q}\notin cl(C)$. There is $o(D)$ containing $\mathcal{Q}$
and disjoint with $C$. Hence $D\subset X\setminus C$ resulting in $C\subset X\setminus D$. Consequently, $X\setminus D\in Q$ contradicting $\mathcal{Q}\in o(D)$.

2. If $C\notin \mathcal{Q}$, then $\mathcal{Q}\in o(X\setminus C)$. However,
$C\cap o(X\setminus C)=\emptyset$ contradicting $\mathcal{Q}\in cl(C)$.
\end{proof}

\begin{Corollary}
$C\subset o(C)\subset cl(C)$ for all $C\subset X$.
In particular, $cl(C)=cl(o(C))$.
\end{Corollary}
\begin{proof}
If $Q\in o(C)$, then $X\setminus C\notin Q$, hence $C\in Q$.
By \ref{ClosureCLemma}, $Q\in cl(C)$.
\end{proof}

\begin{Lemma}\label{DisjointClosuresLemma}
1. If $\bigcap\limits_{i=1}^k cl(C_i)\subset X\setminus \partial (\omega)$, then $\omega(C_1,\ldots,C_k)=0$.\\
2. If $C_i$, $i\leq k$, are closed in $X$, $\bigcap\limits_{i=1}^k C_i=\emptyset$, and $\omega(C_1,\ldots,C_k)=0$,
then $\bigcap\limits_{i=1}^k cl(o(C_i))\subset X\setminus \partial (\omega)$.
\end{Lemma}
\begin{proof}
1. Suppose $\bigcap\limits_{i=1}^k cl(C_i)\subset X\setminus \partial (\omega)$ but $\omega(C_1,\ldots,C_k)=\infty$.
Hence all $C_i$ are $\omega$-unbounded and
there is $\mathcal{Q}\in \partial (\omega)$ containing all $C_i$, $i\leq k$.
By Lemma \ref{ClosureCLemma}, $\mathcal{Q}\in \bigcap\limits_{i=1}^k cl(C_i)$, a contradiction.

2. Suppose $\mathcal{Q}\in \bigcap\limits_{i=1}^k cl(o(C_i))$, $\bigcap\limits_{i=1}^k C_i=\emptyset$, and $\omega(C_1,\ldots,C_k)=0$. Therefore $\mathcal{Q}$ cannot contain all $C_i$'s, so assume
$C_1 \notin \mathcal{Q}$. Hence $\mathcal{Q}\in o(X\setminus C_1)$. However,
$o(C_1)\cap o(X\setminus C_1)=\emptyset$, contradicting $\mathcal{Q}\in cl(o(C_1))$.
\end{proof}

\begin{Proposition}\label{FiniteCoveringProposition}
Supose $D\subset X$ is closed. If $\{C_i\}_{i=1}^k$ is a finite family of subsets of $X$, then the following conditions are equivalent:\\
1. $\bigcup\limits_{i=1}^k o(C_i)$ contains $cl(D)\cap\partial(\omega)$.\\
2. $\omega(D,X\setminus C_1,\ldots,X\setminus C_k)=0$ and 
$D\setminus \bigcup\limits_{i=1}^k C_i\in B(\omega)$.\\
3. $\omega(D,X\setminus C_1,\ldots,X\setminus C_k)=0$.
\end{Proposition}
\begin{proof}
Notice $Q\in cl(D)\cap \partial(\omega)\setminus \bigcup\limits_{i=1}^k o(C_i)$ if and only if $D, X\setminus C_1,\ldots,X\setminus C_k\in Q$
(see \ref{ClosureCLemma}),
so $cl(D)\cap \partial(\omega)\setminus \bigcup\limits_{i=1}^k o(C_i)=\emptyset$ if and only if $\omega(D,X\setminus C_1,\ldots,X\setminus C_k)=0$. Therefore $1)\iff 3)$.

$2)\iff 3)$ follows from the fact that $\omega(D,X\setminus C_1,\ldots,X\setminus C_k)=0$ implies $\omega(D\setminus \bigcup\limits_{i=1}^k C_i)=0$ by \ref{OmegaInequalityOne}.
\end{proof}

\begin{Corollary}\label{FiniteSubfamilyCorollary}
Given a covering $\{o(C)\}_{C\in \mathcal{S}}$ of $cl(D)$, $D\subset X$ being closed, there is a finite subfamily $\mathcal{F}$ of $\mathcal{S}$
covering $cl(D)\cap \partial(\omega)$
such that $D\setminus \bigcup\limits_{C\in \mathcal{F}}C$ belongs to $\mathcal{B}(\omega)$.
\end{Corollary}
\begin{proof}
Suppose $\{o(C)\}_{C\in \mathcal{S}}$ is a cover of $cl(D)$
that has no finite subcover of $cl(D)\cap\partial(\omega)$.
By \ref{FiniteCoveringProposition} it means
$\omega(D,X\setminus C_1,\ldots,X\setminus C_k)=\infty$
for all finite subfamilies $\{C_i\}_{i=1}^k$ of $\mathcal{S}$.
Therefore there is $Q\in\partial(\omega)$ containing $D$ and all sets $X\setminus C$, $C\in \mathcal{S}$. By \ref{ClosureCLemma},
$Q\in cl(D)$. Since $Q\notin \bigcup \{o(C) | {C\in \mathcal{S}}\}$, we arrive at a contradiction.

\end{proof}

\begin{Corollary}\label{LSCompactCorollary}
$X\cup \partial (\omega)$ is large scale compact with respect to
$\mathcal{B}(\omega)$.
\end{Corollary}

\begin{Proposition}\label{ConditionForFormCompactificationToBeT2}
Suppose $\omega$ is a form on $X$. The following conditions are equivalent:\\
a. $X\cup \partial (\omega)$ is large scale compact with respect to
$\mathcal{B}(\omega)$ and Hausdorff.\\
b. For every two different points
$\mathcal{Q}$ and $\mathcal{R}$ of $\partial(\omega)$ there exist $\omega$-closed subsets
$C$ and $D$ of $X$ such that $C\cup D=X$, $C\notin \mathcal{Q}$,
and $D\notin \mathcal{R}$. 
\end{Proposition}
\begin{proof}
a)$\implies$b). Choose disjoint sets $o(U)$ and $o(W)$, where
$Q\in o(U)$, $R\in o(W)$, and $U,W$ are $\omega$-open. 
Therefore $U\cap W=\emptyset$ resulting in $C:=X\setminus U$,
$D:=X\setminus W$ satisfying the required conditions.
\\
b)$\implies$a). 
Since every point of $X\setminus\partial(\omega)$ is open-closed, it suffices to show that every two distinct points $\mathcal{Q},\mathcal{R}\in\partial(\omega)$
have disjoint neighborhoods.
Choose $\omega$-closed subsets
$C$ and $D$ of $X$ such that $C\cup D=X$, $C\notin \mathcal{Q}$,
and $D\notin \mathcal{R}$. Notice that $\mathcal{Q}\in o(X\setminus C)$,
$\mathcal{R}\in o(X\setminus D)$. Suppose $\mathcal{S}\in o(X\setminus C)\cap o(X\setminus D)$. Therefore $C\notin \mathcal{S}$ and $D\notin \mathcal{S}$ contradicting $X=C\cup D\in \mathcal{S}$.
\end{proof}

\begin{Proposition}
If $(X,\omega)$ is a formed set, then $\omega(C_1,\ldots,C_k)=0$
iff $\bigcap\limits_{i=1}^k cl(C_i)\subset X\setminus\partial(\omega)$.
\end{Proposition}
\begin{proof}
$\omega(C_1,\ldots,C_k)=\infty$ if and only if there is $\mathcal{Q}\in \partial(\omega)$ containing all of $C_i$. That is equivalent to 
$\mathcal{Q}\in \bigcap\limits_{i=1}^k cl(C_i)$.
\end{proof}

\section{Normal forms}
In this section we seek conditions for $X\cup\partial(\omega)$ to be Hausdorff.

\begin{Definition}
A formed set $(X,\omega)$ is \textbf{normal} if
$\omega(C_1,\ldots,C_k)=0$ implies existence of subsets $D_i$, $1\leq i\leq k$, of $X$ such that $\bigcup\limits_{i=1}^k D_i=X$ and $\omega(C_i,D_i)=0$ for each $1\leq i\leq k$.
\end{Definition}

\begin{Lemma}\label{ZeroSetsLemma}
If $\{C_i\}_{i=1}^k$ is a family of zero-sets in a topological space $(X,\mathcal{T})$ whose intersection is empty, then there exists a family of zero-sets $\{D_i\}_{i=1}^k$
satisfying the following conditions:\\
a. $C_i\cap D_i=\emptyset$ for each $i\leq k$,\\
b. $\bigcup\limits_{i=1}^k D_i=X$.
\end{Lemma}
\begin{proof}
Choose continuous functions $f_i:X\to [0,1]$ such that $C_i=f_i^{-1}(0)$ for each $i\leq k$. Notice that $g:=\sum\limits_{i=1}^k f_i$ is positive
and let $g_i:=f_i/g$. Define $D_i$ as $g_i^{-1}[1/k,1]$ and one can easily see that $\{D_i\}_{i=1}^k$ satisfies the required conditions.
\end{proof}

\begin{Corollary}\label{FunctionalFormIsNormal}
The basic functional form $\omega_f(X,\mathcal{T})$ of a topological space $(X,\mathcal{T})$ is normal.
\end{Corollary}
\begin{proof}
Apply \ref{ZeroSetsLemma}.
\end{proof}

\begin{Corollary}\label{BasicFormIsNormalForNormal}
The basic topological form of a $T_1$ space $(X,\mathcal{T})$ is normal if and only if $(X,\mathcal{T})$ is normal.
\end{Corollary}
\begin{proof}
If $(X,\mathcal{T})$ is normal, then its functional form equals the basic topological form, hence is normal by \ref{FunctionalFormIsNormal}.

Suppose $\omega(\mathcal{T})$ is normal and $C, D$ are two disjoint closed subsets of $X$. Thus $\omega(\mathcal{T})(C,D)=0$ and there exist
subsets $C', D'$ of $X$ such that $C'\cup D'=X$, $\omega(\mathcal{T})(C,C')=0$,
and $\omega(\mathcal{T})(D,D')=0$. That implies $C\subset U:=X\setminus cl(C')$, $D\subset W:=X\setminus cl(D')$ and $U\cap W=\emptyset$.
Since $(X,\mathcal{T})$ is $T_1$, it is normal.
\end{proof}

\begin{Proposition}\label{LSFormOfMetrizableIsNormal}
If a large scale space $(X,\mathcal{L})$ is metrizable, then the form $\omega(\mathcal{L})$ induced by $\mathcal{L}$ is normal.
\end{Proposition}
\begin{proof}
Suppose $\omega(C_1,\ldots,C_k)=0$ for some subsets of $X$ and $d$ is a metric on $X$ inducing $\mathcal{L}$. Choose $x_0\in X$ and an increasing function
$f:\mathbb{N}\to \mathbb{N}$ such that for each $n\ge 1$, $\bigcap\limits_{i=1}^k B(C_i,n)\subset B(x_0,f(n))$.

For each $i\leq k$, define $D_i$ as
$$\{x\in X | \mbox{ for all }n\ge 1, \ d(x,x_0) > f(n)\implies d(x,C_i) \ge n\}.$$
If $\omega(C_i,D_i)=\infty$ for some $i\leq k$, then there is $m\ge 1$ such that $B(C_i,m)\cap B(D_i,m)$ is unbounded. In particular, there is $y\in B(C_i,m)\cap B(D_i,m)$ satisfying $d(y,x_0) > f(m)$. Thus, $d(y,C_i) \ge m$
contradicting $y\in B(C_i,m)$.

Suppose $z\in X\setminus \bigcup\limits_{i=1}^k D_i$. Therefore, for each $i\leq k$, there is $m_i$ such that $d(z,x_0) > f(m_i)$ but $d(z,C_i) < m_i$.
Put $M=\max\{m_i | i\leq k\}$. Notice $z\in B(x_0,f(M))$
as $z\in \bigcap\limits_{i=1}^k B(C_i,M)\subset B(x_0,f(M))$ which contradicts
$d(z,x_0) > f(M)$.
\end{proof}

\subsection{Special normal forms}

Suppose a set $X$ has a function $N:2^X\times S\to 2^X$, $S$ being a directed set such that the following conditions are satisfied:\\
1. $N(C,s) \subset N(C,t)$ if $t \leq s$,\\
2. For all $C,D \subset X$ and all $t\in S$, there is $s\ge t$ such that
$N(C\cup D,s)\subset N(C,t)\cup N(D,t)$.

Notice the following defines a form:
$$\omega(N)(C_1,\ldots,C_k)=0\iff \bigcap\limits_{i=1}^k N(C_i,s)=\emptyset \mbox{ for some }s\in S.$$

\begin{Proposition}\label{SpecialNormalForms}
If for each $D\subset X\setminus N(C,s)$ there is $t\in S$, $t \ge s$, such that
$N(D,t)\subset X\setminus N(C,t)$, then $\omega(N)$ is normal.
\end{Proposition}
\begin{proof}
Suppose $\bigcap\limits_{i=1}^k N(C_i,s)=\emptyset$.
Define $D_i:=X\setminus N(C_i,s)$ for $i\leq k$. Obviously,
$\bigcup\limits_{i=1}^k D_i=X$.
Also, for each $i\leq k$ there is $t_i > s$ such that 
$N(D_i,t_i)\subset X\setminus N(C_i,t_i)$ which means $\omega(D_i,C_i)=0$.
\end{proof}

\begin{Corollary}\label{UniformFormIsNormal}
If $(X,\mathcal{U})$ is a uniform space, then the form $\omega(\mathcal{U})$ induced by $\mathcal{U}$ is normal.
\end{Corollary}
\begin{proof}
Let $S$ be the family of uniform covers of $X$.
Define $\mathcal{W} \ge \mathcal{V}$ as
$\mathcal{W}$ refines $\mathcal{V}$.
The function $N:2^X\times S\to 2^X$ is defined as $N(C,\mathcal{V})=st(C,\mathcal{V})$.
Notice $\omega(\mathcal{U})=\omega(N)$ (see \ref{FormOnUniformSpaces}).

Let $\mathcal{W}$ be a uniform cover of $X$ such that $st(\mathcal{W},\mathcal{W})$ refines $\mathcal{V}$.
If $D\subset X\setminus st(C,\mathcal{V})$, then $N(D,\mathcal{W})\cap N(C,\mathcal{W})=\emptyset$. Indeed, suppose $x\in st(C,\mathcal{W})\cap st(D,\mathcal{W})$. Pick $y\in C\cap st(x,\mathcal{W})$ and pick $z\in D\cap st(x,\mathcal{W})$. There is $W_1\in \mathcal{W}$ containing $x, y$ and there
is $W_2\in \mathcal{W}$ containing $x,z$. Pick $V\in \mathcal{V}$ containing
$st(W_1,\mathcal{W})$. Notice $W_1\cup W_2\subset V$ and $y\in V\cap C$. Since
$D\subset X\setminus V$ we arrive at a contradiction: $z\in D\cap V$.
\end{proof}

\begin{Definition}
Suppose $X$ is a locally compact Hausdorff space that is locally compact.
Given $C\subset X$ and a compact subset $K$ of $X$ define $N(C,K)$ as the union of components of $X\setminus K$ that intersect $C$.

We define the order on compact subsets of $X$ as follows: $K \leq K'$ if $K\subset K'$. The form $\omega(N)$ is called the \textbf{Freundenthal form} of $X$.
\end{Definition}

\begin{Proposition}
The Freundenthal form is normal.
\end{Proposition}
\begin{proof}
Observe $N(C\cup D,K)= N(C,K)\cup N(D,K)$ for all $C,D\subset X$ and all $K\subset X$ being compact.
Given $C\subset X$ and $K\subset X$ consider $D:=X\setminus N(C,K)$.
Observe that $N(D,K)\cap N(C,K)=\emptyset$. Apply \ref{SpecialNormalForms}.
\end{proof}

\section{Boundaries of normal forms}

In this section we show $X\cup\partial(\omega)$ is Hausdorff if and only if $\omega$ is normal and $T_1$.

\begin{Lemma}\label{BasicNbhdInNormalFormsLemma}
Suppose $\omega$ is a normal form on $X$. If $D\cap C=\emptyset$ and $\omega(D,C)=0$, then there is an $\omega$-open set $U$ containing $D$,  disjoint with $C$, and satisfying $\omega(C,U)=0$.
\end{Lemma}
\begin{proof}
 \textbf{Claim:} There exist disjoint sets $C'$ and $D'$ such that \\
1. $C'\cup D'=X$,\\
2. $C\subset C'$,\\
3. $D\subset D'$,\\
4. $\omega(C,D')=0$ and $\omega(D,C')=0$.\\ \textbf{Proof of Claim:}\\
Choose sets $E$ and $F$ such that $\omega(C,E)=0=\omega(D,F)=0$
and $E\cup F=X$. Notice $C\cap E, D\cap F\in \mathcal{B}(\omega)$.
Put $C':=C\cup (F\setminus F\cap D)$ and $D':= D\cup (E\setminus E\cap C)$.
$\blacksquare$\\
By induction construct an increasing sequence of sets $D_n$ containing $D$ and disjoint with $C'$  so that $\omega(D_n,X\setminus D_{n+1})=0$.
By \ref{ConstructionOfOpenSetsViaForms}, $U:=\bigcup\limits_{i=1}^\infty D_n$ is $\omega$-open, $D\subset U$, and $U\cap C'=\emptyset$.
Therefore $U\subset D'$ and $\omega(C,U)=0$.
\end{proof}

\begin{Corollary}\label{BasicClosureInNormalFormsLemma}
If $\omega$ is a normal form on $X$,
then $x$ belongs to $\omega$-closure of $C\subset X$ if and only if
$x\in C$ or $\omega(x,C)=\infty$.
\end{Corollary}
\begin{proof}
Suppose $x\notin C$ and $\omega(x,C)=0$. By \ref{BasicNbhdInNormalFormsLemma}, $x\notin cl(C)$.

If $x\notin cl(C)$, then $x\in U:=X\setminus cl(C)$
and $\omega(x,X\setminus U)=0$ as $U$ is $\omega$-open.
Since $C\subset X\setminus U$, $\omega(x,C)=0$.
\end{proof}

\begin{Corollary}\label{OmegaOnClosuresGeneralCase}
If $\omega$ is a normal form on $X$ and $\omega(C_1,\ldots,C_k)=0$,
then 
$$\omega(cl(C_1),\ldots,cl(C_k))=0.$$
\end{Corollary}
\begin{proof}
\textbf{Claim:} If $\omega(C,D)=0$,
then $\omega(cl(C),cl(D))=0$.\\
\textbf{Proof of Claim:}
It is sufficient to show $\omega(cl(C),D)=0$.
Choose sets $C'$ and $D'$ such that $C'\cup D'=X$,
$\omega(C,C')=0$, and $\omega(D,D')=0$.
If $x\in D'\setminus C$, then $x\notin cl(C)$ by \ref{BasicClosureInNormalFormsLemma}. Thus,
$cl(C)\subset C\cup D'$.
Since $\omega(C\cup D',D)=0$, the proof of Claim is completed. $\blacksquare$\\
Suppose $\omega(cl(C_1),\ldots,cl(C_k))=\infty$
and choose $D_i\subset X$, $i\leq k$, such that $\omega(C_i,D_i)=0$ for all $i\leq k$, and $\bigcup\limits_{i=1}^k D_i=X$.
There is $Q\in\partial(\omega)$ containing each $cl(C_i)$, $i\leq k$.
As $\omega(cl(C_i),D_i)=0$ for each $i$, none of $D_i$'s belong to $Q$,
a contradiction as $\bigcup\limits_{i=1}^k D_i=X$.
\end{proof}

\begin{Theorem}\label{MainTheoremOnNormalityAndCompactness}
$X\cup \partial (\omega)$ is large scale compact with respect to $\mathcal{B}(\omega)$ and Hausdorff if and only if $\omega$ is normal and $T_1$.
\end{Theorem}
\begin{proof}
Suppose $\omega$ is normal. Given $\mathcal{Q}\ne \mathcal{R}$ there is $D\in \mathcal{Q}$
and $C_1,\ldots, C_k\in \mathcal{R}$ such that $\omega(D,C_1,\ldots,C_k)=0$.
Choose $D',D_1,\ldots,D_k\subset X$ such that
$\omega(D,D')=\omega(C_i,D_i)=0$ for $i\leq k$ and $D'\cup\bigcup\limits_{i=1}^k D_i=X$. By \ref{OmegaOnClosuresGeneralCase}
we may assume $D'$ is $\omega$-closed and all $D_i$'s are $\omega$-closed.
Notice $D'\notin \mathcal{Q}$. Also, $\bigcup\limits_{i=1}^k D_i\notin \mathcal{R}$
(as otherwise $D_j\in \mathcal{R}$ for some $j$ contradicting $\omega(C_j,D_j)=0$).
By \ref{LSCompactCorollary} and \ref{ConditionForFormCompactificationToBeT2}, $X\cup \partial (\omega)$ is Hausdorff and large scale compact with respect to $\mathcal{B}(\omega)$.

Suppose $X\cup \partial (\omega)$ is Hausdorff, large scale compact with respect to $\mathcal{B}(\omega)$,
and $\omega(C(1),\ldots,C(k))=0$ for some $C(i)\subset X$.
That means the intersection of closures of $C(i)$ (in $X\cup \partial (\omega)$) is an element $B$ of $\mathcal{B}(\omega)$.
Choose, for each $\mathcal{Q}$, an index $i(\mathcal{Q})$ such that $\mathcal{Q}\notin cl(C(i(\mathcal{Q})))$. Then choose an $\omega$-open $D(\mathcal{Q})$ such that $cl(o(D(\mathcal{Q})))\cap cl(C(i(\mathcal{Q})))=\emptyset$ and $\mathcal{Q}\in o(D(\mathcal{Q}))$.
Notice $\omega(D(\mathcal{Q}),cl(C(i(\mathcal{Q})))=0$ as otherwise there is $\mathcal{R}\in cl(o(D(\mathcal{Q})))\cap cl(C(i(\mathcal{Q})))$.

There are finitely many points $\mathcal{Q}(j)$ so that
$B':=(X\cup \partial (\omega))\setminus \bigcup\limits_{j=1}^n o(D(\mathcal{Q}(j)))\in B(\omega)$.
For each index $s\leq k$ define $E_s$ as the union of
all $D(\mathcal{Q}(j))$ so that $i(\mathcal{Q}(j))=s$. Notice $\omega(B'\cup B\cup E_s,C_s)=0$
for each $s\leq k$ and $\bigcup\limits_{s=1}^k (B'\cup B\cup E_s)=X$.
\end{proof}

\begin{Lemma}\label{NormalFormsOn2Vectors}
Suppose $\omega_1$, $\omega_2$ are normal forms on $X$. If $\omega_1(C_1,C_2)=\omega_2(C_1,C_2)$ for all $2$-vectors $(C_1,C_2)$ in $X$, then $\omega_2=\omega_1$.
\end{Lemma}
\begin{proof}
Notice $\omega_1(C)= \omega_1(C,C)=\omega_2(C,C)=\omega_2(C)$
for all $C\subset X$.
Now, it suffices to show $\omega_1(C_1,\ldots,C_k)=0$ implies $\omega_2(C_1,\ldots,C_k)=0$ for all $k\ge 3$ and all $k$-vectors $V=(C_1,\ldots,C_k)$ of $X$.
Choose subsets $D_i$, $1\leq i\leq k$, of $X$ such that $\bigcup\limits_{i=1}^k D_i=X$ and $\omega_1(C_i,D_i)=0$ for each $1\leq i\leq k$
which implies $\omega_2(C_i,D_i)=0$ for each $1\leq i\leq k$.
Therefore $\omega_2(V\ast D_i)=0$ for each $1\leq i\leq k$.
Finally, $\omega_2(V)=\omega_2(V\ast X)=\sum\limits_{i=1}^k\omega_2(V\ast D_i)=0$.
\end{proof}

\begin{Theorem}\label{StabilityOfFormsThm}
Suppose $\omega_1$, $\omega_2$ are normal forms on $X$ and
$(X,\mathcal{T}(\omega_1),\mathcal{B}(\omega_1))$ is large scale compact and Hausdorff. If $\mathcal{T}(\omega_2)=\mathcal{T}(\omega_1)$
and $\mathcal{B}(\omega_2)=\mathcal{B}(\omega_1)$, then $\omega_2=\omega_1$.
\end{Theorem}
\begin{proof}
Notice $\mathcal{B}(\omega_2)=\mathcal{B}(\omega_1)$ means exactly
that $\omega_1(C)=\omega_2(C)$ for all $C\subset X$, so our next goal is to prove $\omega_1(C_1,C_2)=\omega_2(C_1,C_2)$ for all $2$-vectors $(C_1,C_2)$ in $X$. It suffices to show $\omega_1(C_1,C_2)=0$ implies $\omega_2(C_1,C_2)=0$ for disjoint subsets $C_1$, $C_2$ of $X$ (remove their intersection if it is not empty). Enlarge $C_2$ to an open subset $U$ of $X$
so that $\omega_1(C_1,U)=0$ (see \ref{BasicNbhdInNormalFormsLemma})
and put $D:=cl(C_1)$. Now, $\omega_1(D,U)=0$ by \ref{OmegaOnClosuresGeneralCase}. Given $x\in D\setminus U$, there is a neighborhood $W(x)$ of $x$ satisfying $\omega_2(W(x),U)=0$
as $x\notin U$ and $U\in \mathcal{T}(\omega_2)$ (use \ref{BasicNbhdInNormalFormsLemma}). There is a finite subset $F$ of $D\setminus U$ for which the complement $E$ of $\bigcup\limits_{x\in F} W(x)\cup (X\setminus D)\cup U$ belongs to $\mathcal{B}(\omega_2)$.
Put $W:=\bigcup\limits_{x\in F} W(x)$ and notice 
Notice $D\subset E\cup W\cup (D\setminus U)$. Since $\omega_2(Y,U)=0$ for each term $Y$ of $ E\cup W\cup (D\setminus U)$, $\omega_2(D,U)=0$
which implies $\omega_2(C_1,C_2)=0$.

To complete the proof of \ref{StabilityOfFormsThm} invoke \ref{NormalFormsOn2Vectors}.
\end{proof}

\begin{Theorem}\label{UniquenessOfNormalFormsOnCompactSpaces}
If $(X,\mathcal{T},\mathcal{B})$ is a large scale compact Hausdorff topological space, then there is a unique normal form $\omega$ inducing $\mathcal{T}$ and having $\omega$-bounded sets identical with the bornology $\mathcal{B}$.
That form is given by the formula
$$\omega(C_1,\ldots,C_k)=0 \iff X_0\cap cl(C_1)\cap\ldots\cap cl(C_k)=\emptyset$$
where $X_0:=X\setminus \bigcup \mathcal{B}$.
\end{Theorem}
\begin{proof}
$\omega$ is normal by an argument similar to the proof of \ref{BasicFormIsNormalForNormal}. 
First, let's show $\mathcal{B}=\mathcal{B}(\omega)$.
If $B\in \mathcal{B}$, then $B$ is closed and $\omega(B)=0$ as $X_0\cap B=\emptyset$. Conversely, suppose $X_0\cap cl(C)=\emptyset$.
For each $x\in cl(C)$ consider its neighborhood $W(x)=\{x\}$. Pick a finite subset $F$ of $cl(C)$ so that the complement
$B$ of $\bigcup\limits_{x\in F}\{x\}\cup (X\setminus cl(C))$ belongs to $\mathcal{B}$. Put $W:=\bigcup\limits_{x\in F}\{x\}\in \mathcal{B}$
and notice $cl(C)\subset W\cup B\in \mathcal{B}$.

Now, let's show $\mathcal{T}=\mathcal{T}(\omega)$.
If $x\in U$ and $U\in \mathcal{T}$, then $\omega(x,X\setminus U)=0$
as $X_0\cap \{x\}\cap (X\setminus U)=\emptyset$.
Thus $U\in \mathcal{T}(\omega)$.
Conversely, if $W\in \mathcal{T}(\omega)$, then for each $x\in W$,
$\omega(x,X\setminus W)=0$ which means $\emptyset=X_0\cap \{x\}\cap cl(X\setminus W)$. Now, $x\in cl(X\setminus W)$ is possible only if $x\notin X_0$, in which case $\{x\}$ is open-closed, so $cl(X\setminus W)\subset (X\setminus \{x\})$, a contradiction. Thus, $W\cap cl(X\setminus W)=\emptyset$
and $W\in \mathcal{T}$.

To conclude the proof of \ref{UniquenessOfNormalFormsOnCompactSpaces}  apply \ref{StabilityOfFormsThm}.
\end{proof}

\begin{Proposition}\label{RelativeTopologyAndForms}
If $(X,\omega)$ is a formed set and $A\subset X$ is closed in the topology $\mathcal{T}$ induced by $\omega$, then the topology on $A$ induced by $\mathcal{T}$ coincides with the topology induced by $\omega_A:=\omega |A$.
Moreover, if $\omega$ is normal and $T_1$, then so is $\omega_A$ and the natural function $i:A\cup \partial(\omega_A)\to X\cup \partial(\omega)$ is a topological embedding.
\end{Proposition}
\begin{proof}
If $U\subset X$ is open, then for each $x\in A\cap U$ one has
$\omega_A(x,A\setminus U)=0$, so $U\cap A$ is open in the topology induced by $\omega_A$. Conversely, if $U\subset A$ has the property
that $\omega_A(x, A\setminus U)=0$ for all $x\in U$,
then (as $A$ is closed in $X$) $\omega(x,X\setminus (U\cup (X\setminus A))=\omega(x,A\setminus U)=0$ for all $x\in U\cup (X\setminus A)$, so $U\cup (X\setminus A)$ is open in $X$ and $U$ is open in the relative topology on $A$.

$i$ is the identity on $A$ and $i(\mathcal{P})$, $\mathcal{P}\in \partial(A,\omega_A)$, consists of all subsets $C$ of $X$ such that $C\cap A\in \mathcal{P}$.
The complement of the image of $i$ is exactly $o(X\setminus A)$.
Given an open subset $U$ of $X$, $i^{-1}(o(U))=o(U\cap A)$, so $i$ is continuous. Also, given an open subset $U$ of $A$,
$i(o(U))=im(i)\cap o(U\cup (X\setminus A))$, so $i$ is an embedding.
\end{proof}

\section{Orthogonality relations}

In this section we describe a geometrical relation between subsets of a set $X$ that can be used to generate a form on $X$.

\begin{Definition}
An \textbf{orthogonality relation} on subsets of a set $X$ is a symmetric relation $\perp$ satisfying the following properties:\\
1. $\emptyset \perp X$,\\
2. $A\perp (C\cup C')\iff A\perp C$ and $A\perp C'$.
\end{Definition}


\begin{Observation}
One can reduce the number of axioms by dropping symmetry and replacing Axiom 2 by\\
2'. $A\perp (C\cup C')\iff C\perp A$ and $C'\perp A$.
\end{Observation}

\begin{Example}\label{OrthInducedByBornology}
For every bornology $\mathcal{B}$ on a set $X$ the relation $A\perp C$ defined as $A\cap C\in \mathcal{B}$ is an orthogonality relation.
\end{Example}

\subsection{Bounded sets}
\begin{Definition}
Given an orthogonality relation $\perp$ on subsets of $X$, a \textbf{bounded subset} $B$ of $X$ is one that is orthogonal to the whole set: 
$$B\perp X.$$
\end{Definition}


 \begin{Definition}
An orthogonality relation $\perp$ on subsets of $X$ is \textbf{small scale} if the empty set is the only  subset of $X$ that is orthogonal to itself. In particular, the only bounded subset of $X$ is the empty set.
\end{Definition}

\begin{Definition}
An orthogonality relation $\perp$ on subsets of $X$ is \textbf{large scale} if each point is a bounded subset of $X$.
\end{Definition}


\subsection{Examples of small scale orthogonality}
\begin{Example}
1. \textbf{Set-theoretic orthogonality}: Disjointness,\\
2. \textbf{Topological orthogonality}: Disjointness of closures,\\
3. \textbf{Metric orthogonality}: Disjointness of $r$-balls for some $r > 0$,\\
4. \textbf{Uniform orthogonality}: Disjointness of $\mathcal{U}$-neighborhoods for some uniform cover $\mathcal{U}$.
\end{Example}


\subsection{Examples of large scale orthogonality}
\begin{Example}
1. \textbf{Set-theoretic large scale orthogonality}: Finiteness of intersection,\\
2. \textbf{Metric large scale orthogonality}: Boundedness of intersection of $r$-balls for all $r > 0$,\\
3. \textbf{Group large scale orthogonality}: Finiteness of $(A\cdot F)\cap (C\cdot F)$ for all finite subsets $F$ of a group $G$.

Same as metric ls-orthogonality for word metrics if $G$ is finitely generated.\\
4. \textbf{Topological ls-orthogonality}: Disjointness of coronas of closures in a fixed compactification $\bar X$ of $X$.
\end{Example}


\subsection{Hyperbolic orthogonality}
Given a metric space $(X,d)$, the \textbf{Gromov product} of $x$ and $y$ with respect to $a\in X$ is defined by
$$
\left< x,y\right>_a=\frac{1}{2}\big(d(x,a)+d(y,a)-d(x,y)\big).
$$

Recall that metric space $(X,d)$ is (Gromov) $ \delta-$\textbf{hyperbolic} if it satisfies the $\delta/4$-inequality:
$$
\left< x,y\right>_{a} \geq \min \{\left< x,z\right>_{a},\left< z,y \right> _{a}\}-\delta/4, \quad \forall x,y,z,a\in X.$$

$(X,d)$ is \textbf{Gromov hyperbolic} if it is $ \delta-$hyperbolic for some $\delta > 0$.

\begin{Definition}
Two subsets $A$ and $C$ of a hyperbolic space $X$ are \textbf{hyperbolically orthogonal} if there is $r > 0$ such that
$$\left< a,c\right>_{p} < r$$
for some fixed $p$ and all $(a,c)\in A\times C$.
\end{Definition}


\subsection{Freundenthal orthogonality}
\begin{Definition}
Suppose $X$ is a locally compact and locally connect topological space.
Two subsets $A$ and $C$ of $X$ are \textbf{Freundenthal orthogonal}
if there is a compact subset $K$ of $X$ such that the union of all components of $X\setminus K$ intersecting $A$ is disjoint from the union of all components of $X\setminus K$ intersecting $C$.
\end{Definition}


\subsection{Normal orthogonality relations}

\begin{Definition}\label{SpanDef}
Given an orthogonality relation $\perp$ on subsets of $X$, two subsets $C$ and $D$ \textbf{$\perp$-span} $X$ if the following conditions are satisfied:\\
1. $C\perp D$,\\
2. $X$ can be decomposed as $X=C'\cup D'$, where $C'\perp D$ and $D'\perp C$.
\end{Definition}

\begin{Remark}
Obviously, we may interpret the word "decompose" in the definition above as $C'\cap D'=\emptyset$ since $D'$ can be replaced by $X\setminus C'$.
The other extreme is when $C\subset C'$ and $D\subset D'$ which can be accomplished by replacing $C'$ by $C\cup C'$ and replacing $D'$ by $D\cup D'$. In that case we may think of $C$ being parallel to $C'$, $D$ being parallel to $D'$ and interpret Definition \ref{SpanDef} as an analog of \textbf{parallel-perpendicular decomposition} in Linear Algebra.
\end{Remark}

\begin{Proposition}
If $C$ and $D$ $\perp$-span $X$, then $C\cap D$ is $\perp$-bounded.
\end{Proposition}
\begin{proof}
$X=C'\cup D'$, where $C'\perp D$ and $D'\perp C$.
Therefore $C'\perp (C\cap D )$ and $D'\perp (C\cap D )$
resulting in $(C\cap D)\perp X$.
\end{proof}

\begin{Definition}
An orthogonality relation $\perp$ on subsets of $X$ is \textbf{normal} if
 $C$ and $D$ $\perp$-span $X$ whenever $C\perp D$.
\end{Definition}

\begin{Example}
The topological orthogonality relation on a topological space is normal if $X$ is topologically normal.
\end{Example}


\begin{Definition}
The \textbf{functional orthogonality relation} $\perp$ on a topological space $X$ is defined as follows: $C\perp D$ if there is a continuous function $f:X\to [0,1]$ such that
$f(C)\subset \{0\}$ and $f(D)\subset \{1\}$.
\end{Definition}

\begin{Proposition}
The functional orthogonality relation $\perp$ on a topological space $X$ is always normal.
\end{Proposition}
\begin{proof}
For any continuous $f:X\to [0,1]$ satisfying
$f(C)\subset \{0\}$, $f(D)\subset \{1\}$, one puts $C'=f^{-1}[0.5,1]$, $D'=f^{-1}[0,0.5]$
and observe $X=C'\cup D'$, $C'\perp C$, and $D'\perp D$.
\end{proof}

\begin{Proposition}
The hyperbolic orthogonality relation $\perp$ on a Gromov hyperbolic space is normal.
\end{Proposition}
\begin{proof}
Suppose $A$ and $C$ are subsets of a pointed Gromov hyperbolic space $(X,p)$ such that $\left< a,c\right>_{p} < r$
 for all $(a,c)\in A\times C$. Define $A'$ as $\{x\in X | \left< a,x\right>_{p} < r+\delta \ \forall a\in A\}$.
Define $C'$ as $\{x\in X | \left< c,x\right>_{p} < r+\delta \ \forall c\in C'\}$.
Clearly, $A'\perp A$ and $C'\perp C$, so it remains to show $A'\cup C'=X$.
Suppose there is $x\in X\setminus (A'\cup C')$. There is $a\in A$ and $c\in C$
such that $ \left< a,x\right>_{p} \ge r+\delta$ and $ \left< c,x\right>_{p} \ge r+\delta$.
Therefore
$$
\left< a,c\right>_{p} \geq \min \{\left< x,a\right>_{p},\left< c,x \right> _{p}\}-\delta/4 > r$$
a contradiction.
zzz
\end{proof}

\subsection{Proximity spaces}
There is a more general structure than uniform spaces, namely a proximity (see \cite{NaWa}). In this section we show that those structures correspond to normal small scale orthogonal relations.
\begin{Definition}
A \textbf{proximity space} $(X, \delta)$ is a set $X$ with a relation $\delta$ between subsets of $X$ satisfying the following properties:\\
For all subsets $A, B$ and $C$ of $X$\\
1. $A \delta B \implies B \delta A$\\
2. $A \delta B \implies A \ne\emptyset$\\
3. $A\cap B\ne\emptyset \implies A \delta B$\\
4. $A \delta (B\cup C) \iff (A \delta B \mbox{ or } A \delta C)$\\
5. $\forall E, A \delta E \mbox{ or }B \delta (X\setminus E) \implies A \delta B$.
\end{Definition}

\begin{Proposition}
Normal small scale orthogonality relations are in one-to-one correspondence with proximity relations.
\end{Proposition}
\begin{proof}
Given a small scale orthogonal relation $\perp$ we define $A\delta C$ as $\lnot (A\perp C)$.

Conversely, given a proximity relation $\delta$ we define $A\perp C$ as $\lnot (A\delta C)$.

The proof amounts to negating implications, so let's show only the implication $A\cap B\ne\emptyset \implies A \delta B$. If it fails, then we have two orthogonal sets $A$ and $B$ with non-empty intersection $A\cap B$. However, in this case $A\cap B$ is self-orthogonal, a contradiction.
\end{proof}

\section{$\perp$-continuous functions}


\begin{Definition}
Given two sets $X$ and $Y$ equipped with orthogonality relations $\perp_X$ and $\perp_Y$, a function $f:X\to Y$ is \textbf{$\perp$-continuous}
if 
$$f(A)\perp_Y f(C)\implies A\perp_X C$$
for all subsets $A, C$ of $X$.
\end{Definition}


\subsection{Small Scale Examples}
In the small scale $\perp$-continuous functions are exactly neighborhood-continuous functions with respect to the induced neighborhood operator. Therefore both examples below follow from \cite{DW} in view of \ref{PerpContinuousVsNbhdCont}.
\begin{Example}\label{TopContinuousVsPerpContinuous}
If both $X$ and $Y$ are normal spaces equipped with topological orthogonality relations, then $\perp$-continuity is ordinary \textbf{topological continuity}.
\end{Example}

\begin{Example}\label{UniformContinuityVsPerpContinuity}
If both $X$ and $Y$ are uniform spaces equipped with uniform orthogonality relations, then $\perp$-continuity is ordinary \textbf{uniform continuity}.
\end{Example}

\subsection{Large Scale Examples}
\begin{Example}
If both $X$ and $Y$ are metric spaces equipped with metric $ls$-orthogonality relations
and $f:X\to Y$ preserves bounded sets, then $\perp$-continuity is the same as $f$ being \textbf{coarse and bornologous}.
\end{Example}
\begin{proof}
Recall that $f:X\to Y$ is bornologous if, for each $r > 0$, there is $s > 0$ such that
$\diam(f(A)) < s$ if $\diam(A) < r$. 

Notice that every $\perp$-continuous function co-preserves bounded sets, i.e. it is coarse.
Suppose $f$ is $\perp$-continuous but not bornologous. Hence, there is a sequence
$B_n$ of uniformly bounded subsets of $X$ whose images $f(B_n)$ have diameters diverging to infinity. We may reduce it to the case of each $B_n$ consisting of exactly two points $x_n$ and $y_n$ so that both $f(x_n)$ and $f(y_n)$ diverge to infinity.
Notice $A:=\{f(x_n)\}_{n\ge 1}$ and $C:=\{f(y_n)\}_{n\ge 1}$ are orthogonal in $Y$
but their point-inverses are not orthogonal in $X$, a contradiction.

Suppose $f$ is coarse and bornologous but not $\perp$-continuous. Choose two orthogonal subsets
$A$ and $C$ of $Y$ whose point-inverses are not orthogonal.
Therefore the intersection of $B(f^{-1}(A),r)$ and $B(f^{-1}(C),r)$ is unbounded
for some $r > 0$ and the image of that intersection is unbounded. There is $s > 0$ satisfying
$f(B(Z,r))\subset B(f(Z),s)$ for all subsets $Z$ of $X$.
Therefore, the intersection of $B(A,s)$ and $B(C,s)$ is unbounded, a contradiction.
\end{proof}

\begin{Example}\label{SlowlyOscillatingExampleOfPerpContinuous}
If $X$ is a metric space equipped with metric ls-orthogonality relation and $Y$ is a compact metric space equipped with small scale metric orthogonality, then $\perp$-continuity is the same as $f$ being \textbf{slowly oscillating}.
\end{Example}
\begin{proof}
Recall that $f:X\to Y$ is slowly oscillating if, for every pair of sequences $\{x_n\}_{n\ge 1}$, $\{y_n\}_{n\ge 1}$ in $X$, $\lim\limits_{n\to\infty}d_Y(f(x_n),f(y_n))=0$
if $\{d_X(x_n,y_n)\}_{n\ge 1}$ is uniformly bounded and $d(x_n,x_1)\to\infty$ as $n\to\infty$.

Suppose $f$ is $\perp$-continuous but not slowly oscillating. Hence, there is 
pair of sequences $\{x_n\}_{n\ge 1}$, $\{y_n\}_{n\ge 1}$ in $X$, 
and $\epsilon > 0$ such that $d_Y(f(x_n),f(y_n)) > \epsilon$ for each $n\ge 1$
and $\{d_X(x_n,y_n)\}_{n\ge 1}$ is uniformly bounded.
We may assume that the limit of $f(x_n)$ is $z_1$, the limit of $f(y_n)$ is $z_2$.
In particular $d_Y(z_1,z_2)\ge \epsilon$.
The sets $B(z_1,\epsilon/3)$ and $B(z_2,\epsilon/3)$ are orthogonal in $Y$
but their point-inverses in $X$ are not, a contradiction.

Suppose $f$ is slowly oscillating but not $\perp$-continuous. Choose two orthogonal subsets
$A$ and $C$ of $Y$ whose point-inverses are not orthogonal.
Therefore the intersection of $B(f^{-1}(A),r)$ and $B(f^{-1}(C),r)$ is unbounded
for some $r > 0$. 
Hence there are two sequences diverging to infinity in $X$: $\{x_n\}_{n\ge 1}$
in $f^{-1}(A)$ and $\{y_n\}_{n\ge 1}$
in $f^{-1}(C)$ such that $d_X(x_n,y_n) < 2r$ for each $n$.
Consequently, $\lim\limits_{n\to\infty}d_Y(f(x_n),f(y_n))=0$
contradicting orthogonality of $A$ and $C$.
\end{proof}

\subsection{Quotient structures}

It is well-known that defining quotient maps in both the uniform category and in the coarse category is tricky. In contrast, in sets equipped with orthogonality relations it is quite easy.

\begin{Definition}
Suppose $\perp_X$ is an orthogonality relation on a set $X$. Given a surjective function $f:X\to Y$ define $C\perp_Y D$ to mean $f^{-1}(C)\perp_X f^{-1}(D)$.
\end{Definition}


It is easy to check that $\perp_Y$ is an orthogonality relation on $Y$, called the \textbf{quotient orthogonality relation}. Also, it is clear that the following holds:

\begin{Proposition}
Suppose $\perp_X$ is an orthogonality relation on a set $X$, $f:X\to Y$ is a surjective function, and $Y$ is equipped with the quotient orthogonality relation $\perp_Y$.
Given any $\perp$-continuous $h:X\to Z$ that is constant on fibers of $f$,
there is unique $\perp$-continuous $g:Y\to Z$ such that $h=g\circ f$.
\end{Proposition}

\section{Neighborhood operators}

This section is devoted to exploring the relation between orthogonality relations and neighborhood operators.

\begin{Definition} \cite{DW}
A \textbf{neighborhood operator} $\prec$ on a set $X$ is a relation between its subsets satisfying the following conditions:

\begin{itemize}
\item[$\mathsf{(N0)}$] $A \prec X$ for all $A \subseteq X$.
\item[$\mathsf{(N1)}$] if $A \prec B$ then $X \setminus B \prec X \setminus A$.
\item[$\mathsf{(N2)}$] if $A \prec B \subseteq C$, then $A \prec C$.
\item[$\mathsf{(N3)}$] if $A \prec N$ and $A' \prec N'$ then $A \cup A' \prec N \cup N'$.
\end{itemize}
\end{Definition}

\begin{Observation}
Note that $\mathsf{(N0)}$ is implied by $\mathsf{(N1)}$ and the condition $X \prec X$. Also, it is easy to see that, together, axioms $\mathsf{(N0)}-\mathsf{(N3)}$ imply:

\begin{itemize}
\item[$\mathsf{(N0')}$] $\varnothing \prec A$ for all $A \subseteq X$.
\item[$\mathsf{(N2')}$] if $A \subseteq B \prec C$ then $A \prec C$.
\item[$\mathsf{(N3')}$] if $A \prec N$ and $A' \prec N'$ then $A \cap A' \prec N \cap N'$.
\end{itemize}
\end{Observation}

\begin{Definition}
A \textbf{normal neighborhood operator} $\prec$ satisfies the following condition:

\begin{itemize}
\item[$\mathsf{(N4)}$] for every pair of subsets $A \prec C$, there is a subset $B$ with $A \prec B \prec C$.
\end{itemize}
\end{Definition}

\begin{Proposition}
Each orthogonality relation $\perp$ on $X$ induces a neighborhood operator $\prec$ defined as follows: $A\prec U$ if $A\perp X\setminus U$ and $A\subset U$.\\
It is normal if and only if $\perp$ is normal.
\end{Proposition}
\begin{proof}
Left to the reader.
\end{proof}


\begin{Proposition}
Each neighborhood operator $\prec$ on $X$ induces a small scale orthogonality relation $\perp$ 
defined as follows: $A\perp U$ if $A\prec X\setminus U$.\\
$\perp$ is normal if and only if $\prec$ is normal.
\end{Proposition}
\begin{proof}
Left to the reader.
\end{proof}

\begin{Definition}\cite{DW} 
Let $X$ be a set and $\prec$ a neighborhood operator. If $A$ is a subset of $X$, then the \textbf{induced neighbourhood operator} $\prec_A$ on subsets of $A$ is defined as follows: $S \prec_A T$ precisely when there exists a subset $T'$ of $X$ such that $S \prec T'$ as subsets of $X$ and $T = T' \cap A$. 
\end{Definition}

\begin{Proposition}\label{PerpContinuousVsNbhdCont}
Suppose $X$ is a set equipped with an orthogonality relation $\perp_X$ and $Y$ is a set equipped with
a small scale orthogonality relation $\perp_Y$.
A function $f:A\subset X\to Y$ is neighborhood continuous (with respect to the induced neighborhood operators) if and only if it is $\perp$-continuous.
\end{Proposition}
\begin{proof}
Suppose $f:A\subset X\to Y$ is neighborhood continuous and $C\perp_Y D$.
Therefore $C\prec_Y Y\setminus D$ and $f^{-1}(C)\prec_A f^{-1}(Y\setminus D)$.
That means existence of $S\subset X$ such that $S\cap A= f^{-1}(Y\setminus D)$ and $ f^{-1}(C)\prec_X S$.
Consequently,  $f^{-1}(C)\perp_X (X\setminus S)$. Since  $f^{-1}(D)\subset X\setminus S$,
$ f^{-1}(D)\perp_X  f^{-1}(C)$.

Suppose $f:A\subset X\to Y$ is $\perp$-continuous and $C\prec_Y D$.
Hence $C\perp_Y (Y\setminus D)$ and $f^{-1}(C)\perp_X f^{-1}(Y\setminus D)$.
That implies $f^{-1}(C)\prec_X S$, where $S:=X\setminus f^{-1}(Y\setminus D)$.
Since $S\cap A=f^{-1}(D)$, $f$ is neighborhood continuous.
\end{proof}

\begin{Corollary}\label{RealExtensionLemma}
Suppose $X$ is a set equipped with a normal orthogonality relation $\perp_X$
and $[a,b]\subset \mathbb{R}$ is equipped with the topological orthogonality relation $\perp$.
If $f:A\subset X\to [a,b]$ is $\perp$-continuous, then it extends to a $\perp$-continuous
$\bar f:X\to [a,b]$. 
\end{Corollary}
\begin{proof}
In view of \ref{PerpContinuousVsNbhdCont}, it suffices to switch to neighborhood continuity
and that case is done in \cite{DW} (Theorem 8.5).
\end{proof}

\begin{Corollary}\label{ComplexExtensionLemma}
Suppose $X$ is a set equipped with a normal orthogonality relation $\perp_X$
and the set of complex numbers $\mathbb{C}$ is equipped with the topological orthogonality relation $\perp$.
If $f:A\subset X\to \mathbb{C}$ is $\perp$-continuous with metrically bounded image, then it extends to a $\perp$-continuous
$\bar f:X\to \mathbb{C}$ with metrically bounded image. 
\end{Corollary}
\begin{proof}
To apply \ref{RealExtensionLemma} it suffices to show that $g,h:A\to [a,b]$
are $\perp$-continuous if and only $g\Delta h:A\to [a,b]\times [a,b]$, $(g\Delta h)(x):=(g(x),h(x))$,
is $\perp$-continuous.

In one direction it is obvious, so assume $C, D\subset [a,b]\times [a,b]$ are metrically separated.
That means there is $\epsilon > 0$ such that $|z_1-z_2|\ge \epsilon$ if $z_1\in C$ and $z_2\in D$.
Cover $[a,b]\times [a,b]$
by finitely many sets of the form $B_1\times B_2$, where $B_1$ and $B_2$ are intervals of length $\epsilon/4$. Notice $(g\Delta h)^{-1}(C\cap (B_1\times B_2))\perp (g\Delta h)^{-1}(D\cap (B'_1\times B'_2))$ for any choice of $B_1, B_2, B_1', B_2'$. Therefore, $(g\Delta h)^{-1}(C\cap (B_1\times B_2))\perp (g\Delta h)^{-1}(D)$ for any choice of $B_1, B_2$.
Finally, $(g\Delta h)^{-1}(C)\perp (g\Delta h)^{-1}(D)$.
\end{proof}

\begin{Observation}\label{ProductOfPerpFunctions}
Observe that the proof of \ref{ComplexExtensionLemma} can be used to prove that, given two functions $f,g:X\to [0,1]$ from a set equipped with an orthogonality relation $\perp$, the function $h:X\to [0,1]\times [0,1]$ is $\perp$-continuous if and only if both $f$ and $g$ are $\perp$-continuous.

\end{Observation}

\section{Orthogonality relations vs forms}

The purpose of this section is to show that normal orthogonality relations on a set are in one-to-one correspondence with normal forms on $X$.

\begin{Proposition}
Every form $\omega$ on $X$ induces a natural orthogonality relation
$\perp_\omega$ on $X$ defined by $C\perp_\omega D$ if and only if
$\omega(C,D)=0$. If $\omega$ is normal, then $\perp_\omega$
is normal.
\end{Proposition}
\begin{proof}
Left to the reader.
\end{proof}

\begin{Theorem}
Every normal orthogonality relation $\perp$ on a set $X$ extends uniquely to a normal multilinear form $\omega(\perp)$.
\end{Theorem}
\begin{proof}
Define $\omega(\perp)$ as follows: $\omega(\perp)(C_1,\ldots,C_k)=0$ if and only if there exist sets $D_i$, $i\leq k$ such that $D_i\perp C_i$ for each $i\leq k$ and $\bigcup\limits_{i=1}^k D_i=X$.

Suppose $\omega(\perp)(E_0,C_1,\ldots,C_k)=0$ and
$\omega(\perp)(E_0',C_1',\ldots,C_k')=0$. Choose sets
$D_i$, $0\leq k$ such that $D_i\perp C_i$ for each $1\leq i\leq k$,
$D_0\perp E_0$, and $\bigcup\limits_{i=0}^k D_i=X$.
Choose sets
$D_i'$, $0\leq k$ such that $D_i'\perp C_i$ for each $1\leq i\leq k$,
$D_0'\perp E_0'$, and $\bigcup\limits_{i=0}^k D_i'=X$.
Observe $(D_i\cup D_i')\perp C_i$ for each $1\leq i\leq k$,
$(D_0\cap D_0')\perp (E_0\cup E_0')$,
and $(D_0\cap D_0')\cup \bigcup\limits_{i=1}^k (D_i\cup D_i')=X$.
Thus, $\omega(\perp)(E_0\cup E_0',C_1,\ldots,C_k)=0$.

Suppose a normal form $\omega$ induces $\perp$. That means
$\omega(C_1,C_2)=\omega(\perp)(C_1,C_2)$ for all
subsets $C_1, C_2$ of $X$.
By \ref{NormalFormsOn2Vectors}, $\omega=\omega(\perp)$.
\end{proof}

\subsection{Extending hyperbolic orthogonality relations}

\begin{Definition}
Given a Gromov hyperbolic space and $C\subset X$, for every $r > 0$ we define
$$N(C,r):=\{x\in X | \left<x,c\right>_p > r \mbox{ for some } c\in C\}.$$

\end{Definition}

\begin{Lemma}\label{BasicHypLemma}
If $(X,d)$ is a Gromov hyperbolic space, then\\
1. $C_1\cap N(C_2,r)\subset N(C_1,r)\cap N(C_2,r)$ for all $r > 0$.\\
2. If $C_1\cap N(C_2,r)\subset B(p,s)$, then
$N(C_1,r+s+\delta)\cap N(C_2,r+s+\delta)=\emptyset$.
\end{Lemma}
\begin{proof}
1. If $x\in C_1\cap N(C_2,r)$,
then  $\left<x,c_2\right>_p > r$
for some $c_2\in C_2$. Hence $d(x,p) > r$, so $x\in N(C_1,r)\cap N(C_2,r)$.\\
2. If $x\in N(C_1,r+s+\delta)\cap N(C_2,r+s+\delta)$,
then $\left<x,c_1\right>_p > r+s+\delta$ and $ \left<x,c_2\right>_p > r+s+\delta$
for some $(c_1,c_2)\in C_1\times C_2$.
Therefore
$ \left< c_1,c_2\right>_{p} \geq \min \{\left< x,c_1\right>_{p},\left< c_2,x \right> _{p}\}-\delta/4 > r+s$ resulting in $c_1\in C_1\cap N(C_2,r)\setminus B(p,s)$, a contradiction.
\end{proof}

\begin{Corollary}\label{CharHypPerpendicularity}
If $(X,d)$ is a Gromov hyperbolic space with the hyperbolic orthogonality relation $\perp$, then the following conditions are equivalent for $C_1,C_2\subset X$:\\
1. $C_1\perp C_2$,\\
2. There is $r > 0$ such that $C_1\cap N(C_2,r)=\emptyset$,\\
3. There is $r > 0$ such that $C_1\cap N(C_2,r)$ is $d$-bounded,\\
4. There is $r > 0$ such that $N(C_1,r)\cap N(C_2,r)=\emptyset$.
\end{Corollary}
\begin{proof}
1$\implies$2. If $C_1\perp C_2$, then there is $s > 0$ such that
$ \left<c_1,c_2\right>_p  < r$ for all $(c_1,c_2)\in C_1\times C_2$.
Therefore $C_1\cap N(C_2,r)=\emptyset$.\\
2$\implies$3 and 3$\implies$4 follow from \ref{BasicHypLemma}. \\
4$\implies$1. 
Suppose $N(C_1,s)\cap N(C_2,s)=\emptyset$ for some $s > 0$.
If $ \left<c_1,c_2\right>_p > s+\delta$ for some $(c_1,c_2)\in C_1\times C_2$,
then consider the point $x$ on a geodesic from $p$ to $c_1$
at the distance $ \left<c_1,c_2\right>_p$ from $p$.
Notice $\left<c_1,x\right>_p= \left<c_1,c_2\right>_p > s+\delta$
and $ \left<c_2,x\right>_p \geq \min \{\left< x,c_1\right>_{p},\left< c_2,c_1 \right> _{p}\}-\delta/4 > s$, a contradiction.
\end{proof}

\begin{Corollary}\label{CharHypPerp2}
If $(X,d)$ is a Gromov hyperbolic space with the hyperbolic orthogonality relation $\perp$, then the following conditions are equivalent for $C_1,C_2\subset X$:\\
1. $C_1\perp C_2$,\\
2. $C_1\perp N(C_2,r)$ for some $r > 0$. 
\end{Corollary}
\begin{proof}
1$\implies$2. If $C_1\perp C_2$, then 
there is $r > 0$ such that $N(C_1,r)\cap N(C_2,r)=\emptyset$.
By \ref{CharHypPerpendicularity}, $C_1\perp N(C_2,r)$.\\
2$\implies$1. If $C_1\perp N(C_2,r)$ for some $r > 0$,
then $C_1\cap N(C_2,r)$ is $d$-bounded, so $C_1\perp C_2$ by \ref{CharHypPerpendicularity}.
\end{proof}

\begin{Corollary}\label{NormalExtOfHypPerp}
If $(X,d)$ is a Gromov hyperbolic space then, the normal extension $\omega$ of the hyperbolic orthogonality relation $\perp$ is described as follows:\\
$\omega(C_1,\ldots,C_k)=0$ if and only if $\bigcap\limits_{i=1}^k N(C_i,r)=\emptyset$ for some $r > 0$. 
\end{Corollary}
\begin{proof} Since $N(C_1\cup C_2,r)=N(C_1,r)\cup N(C_2,r)$,
$\omega$ is indeed a form. By \ref{CharHypPerpendicularity}, the orthogonality relation induced by $\omega$ agrees with the hyperbolic orthogonality relation.

It remains to show $\omega$ is normal. Use \ref{SpecialNormalForms}
and \ref{BasicHypLemma}.
\end{proof}

\begin{Observation}
Similarly to forms, every orthogonality relation $\perp$ on $X$ induces a topology on $X$ as follows: $U$ is open if and only if $x\perp (X\setminus U)$ for all $x\in U$.
\end{Observation}

\section{Form-continuous functions}

\begin{Definition}
A function $f:(X,\omega_X)\to (Y,\omega_Y)$ of formed sets is \textbf{form-continuous}
if $\omega_Y(f(C_1),\ldots,f(C_k))=0$ implies
$ \omega_X(C_1,\ldots,C_k)=0$ for all subsets $C_i$, $i\leq k$, of $X$.
\end{Definition}

\begin{Proposition}\label{FormContinuousVsPerpContinuous}
If $f:(X,\omega_X)\to (Y,\omega_Y)$ is a function of
formed sets such that $\omega_Y$ is normal, then $f$
is form-continuous if and only if it is $\perp$-continuous with respect to
orthogonality relations induced by $\omega_X$ and $\omega_Y$, respectively.
\end{Proposition}
\begin{proof}
We need to show that $f$ is form-continuous
if and only if for all subsets $C,D$ of $X$ the equality $\omega_Y(f(C),f(D))=0$ implies $\omega_X(C,D)=0$.

The implication in one direction is obvious, so assume
$\omega_Y(f(C_1),\ldots,f(C_k))=0$ but $\omega_X(C_1,\ldots,C_k)=\infty$
for some subsets $C_i$ of $X$. Choose $\mathcal{Q}\in\partial(\omega_X)$ containing all $C_i$'s and choose subsets $D_i$ of $Y$ such that
$\omega_Y(D_i,f(C_i))=0$ for each $i\leq k$ and $\bigcup\limits_{i=1}^k D_i=Y$. There is $j$ so that $\mathcal{Q}$ contains $E:=f^{-1}(D_j)$. Now,
$\omega_Y(f(E),f(C_j)=0$ as $f(E)\subset D_j$ which implies
$\omega_X(E,C_j)=0$, a contradiction to $E,C_j\in \mathcal{Q}$.
\end{proof}

\begin{Theorem}\label{ContExtOfFormContThm}
Suppose If $f:(X,\omega_X)\to (Y,\omega_Y)$ is a form-continuous function of
formed sets.
If $\omega_Y$ is normal and $T_1$, then
the unique continuous extension
$$\tilde f:X\cup \partial (\omega_X)\to Y\cup \partial(\omega_Y)$$
of $f$ is given by the formula
$$\tilde f(\mathcal{Q}):=\{C\subset Y\mid \omega_Y(C,f(D))=\infty\mbox{ for all }D\in \mathcal{Q}\}.$$
In particular, the following statements hold:\\
1. $\tilde f(\mathcal{Q})$ is the unique element of $\partial(\omega_Y)$ containing all $f(D)$, $D\in \mathcal{Q}$,\\
2. $\tilde f(\mathcal{Q})$ is the unique element of the intersection of
closures of $f(D)$ in $Y\cup \partial(\omega_Y)$, where $D\in \mathcal{Q}$.
\end{Theorem}
\begin{proof} By \ref{MainTheoremOnNormalityAndCompactness},
$Y\cup \partial(\omega_Y)$ is normal, so all we have to show is $\tilde f$ is well-defined and continuous.

First of all, we have to make sure that $\tilde f(\mathcal{Q})$ exists. Observe that $\omega_Y(f(E),f(D))=\infty$ for all $E,D\in \mathcal{Q}$, so
$\tilde f(\mathcal{Q})$ contains $f(\mathcal{Q})$.
In particular, it shows $\tilde f(\partial (\omega_X))\subset \partial (\omega_Y)$.

Suppose $\omega(C_1,\ldots,C_k)=0$ for some $C_i\in \tilde f(\mathcal{Q})$. Choose $D_i\subset Y$ such that $\omega(C_i,D_i)=0$ for each $i$ and $\bigcup\limits_{i=1}^k D_i=Y$.
Therefore $E=f^{-1}(D_j)\in Q$ for some $j$ contradicting
$\omega_Y(C_j,f(E))=\infty$.

Suppose $\omega_Y(C,C_1,\ldots,C_k)=\infty$ for all $C_i\in \tilde f(\mathcal{Q})$
but $\omega(C,f(E))=0$ for some $E\in\mathcal{Q}$. Notice $f(E)\in \tilde f(Q)$, a contradiction (put $k=1$ and $C_1=f(E)$). Thus, $\tilde f(Q)\in \partial(\omega_Y)$.

Suppose $\tilde f(\mathcal{Q})\in o(C)$, $C$ being $\omega_Y$-open. Choose $\omega_Y$-open $D\subset Y$ such that
$\tilde f(\mathcal{Q})\in o(D)\subset cl(o(D))\subset o(C)$. In particular, $\omega_Y(D,Y\setminus C)=0$, as otherwise there is $\mathcal{P}\in \partial(\omega_Y)$ containing both $D$ and $Y\setminus C$ which belongs
to $cl(D)$ by \ref{ClosureCLemma} but not to $o(C)$.

Hence $\mathcal{Q}\in o(f^{-1}(D))$ (otherwise $f^{-1}(D)\in \mathcal{Q}$
and $D\in \tilde f(\mathcal{Q})$ which contradicts $\tilde f(\mathcal{Q})\in o(D)$)
and for any $\mathcal{R}\in o(f^{-1}(D))$ we have $\tilde f(\mathcal{R})\in o(C)$.
Indeed, $D\in \tilde f(\mathcal{R})$ (as $f^{-1}(D)\in \mathcal{R}$), so $ Y\setminus C\notin \tilde f(\mathcal{R})$
which means $\tilde f(\mathcal{R})\in o(C)$.

Since $f(D)\in \mathcal{P}$ implies $\mathcal{P}\in cl(f(D))$ (see \ref{ClosureCLemma}), it suffices to show the validity of Statement 2.
Suppose $\mathcal{P}\ne \tilde f(\mathcal{Q})$ and pick $C\in \mathcal{P}$ that does not belong to $\tilde f(\mathcal{Q})$.
Hence, $\omega_Y(C,f(D))=0$ for some $D\in \mathcal{Q}$.
By \ref{OmegaOnClosuresGeneralCase},
$\omega_Y(C,cl_Y(f(D)))=0$. Therefore $\mathcal{P}$ does not belong to the closure
of $f(D)$ in $Y\cup \partial(\omega_Y)$.
\end{proof}

\begin{Proposition}\label{IffExtendsItIsFormContinuous}
Suppose If $f:(X,\omega_X)\to (Y,\omega_Y)$ is a function of
formed sets such that $\omega_Y(f(C))=0$ implies
$ \omega_X(C)=0$ for all subsets $C$ of $X$.
If $f$ has a continuous extension
$$\tilde f:X\cup \partial (\omega_X)\to Y\cup \partial(\omega_Y)$$
then it is form-continuous.
\end{Proposition}
\begin{proof}
Suppose $\omega_Y(f(C_1),\ldots,f(C_k))=0$ but
$ \omega_X(C_1,\ldots,C_k)=\infty$ for some subsets $C_i$, $i\leq k$, of $X$. Choose $\mathcal{Q}\in \partial(\omega_X)$ containing
all $C_i$, $i\leq k$. There is $j$ such that $f(C_j)\notin \tilde f(\mathcal{Q})$.
Hence $\tilde f(\mathcal{Q})\notin cl(f(C_j))$ contradicting $\mathcal{Q}\in cl(C_j)$.
\end{proof}

\begin{Corollary}\label{CechStoneCompactificationCorrollary}
If $(X,\mathcal{T})$ is a completely regular topological space and $\omega_f(\mathcal{T})$ is the basic functional form of $\mathcal{T}$, then $X\cup \partial(\omega_f(\mathcal{T}))$ is the \v Cech-Stone compactification of $(X,\mathcal{T})$.
\end{Corollary}
\begin{proof}
Since $\omega_f$ is small scale, i.e. the only bounded set is $\emptyset$, $X\cup \partial(\omega_f(\mathcal{T}))$ is compact. It is also Hausdorff as $\omega_f$ is normal (see \ref{FunctionalFormIsNormal}) and $T_1$.

Let $Y$ be any compact Hausdorff space equipped with the basic topological form $\omega_Y$.
By \ref{TopContinuousVsPerpContinuous} and \ref{FormContinuousVsPerpContinuous} continuous maps $f:X\to Y$ are form-continuous. By \ref{ContExtOfFormContThm} every continuous function $f:X\to Y$
extends over $X\cup \partial(\omega_f(\mathcal{T}))$
and by \ref{IffExtendsItIsFormContinuous} only continuous functions $f:X\to Y$ extend over $X\cup \partial(\omega_f(\mathcal{T}))$. That's the characteristic property of \v Cech-Stone compactification of $(X,\mathcal{T})$.
\end{proof}

\begin{Corollary}\label{Samuel-SmirnovCompactificationCorrollary}
If $(X,\mathcal{U})$ is a uniform space and $\omega(\mathcal{U})$ is the  induced form on $X$, then $X\cup \partial(\omega(\mathcal{U}))$ is the Samuel-Smirnov compactification of $(X,\mathcal{U})$.
\end{Corollary}
\begin{proof}
Since $\omega(\mathcal{U})$ is small scale (see \ref{FormOnUniformSpaces}), i.e. the only bounded set is $\emptyset$, $X\cup \partial(\omega(\mathcal{U}))$ is compact. It is also Hausdorff as 
$\omega(\mathcal{U})$ is normal (see \ref{UniformFormIsNormal}) and $T_1$.

Let $Y$ be any compact metric space equipped with the basic topological form $\omega_Y$.
By \ref{UniformContinuityVsPerpContinuity} and \ref{FormContinuousVsPerpContinuous} uniformly continuous maps $f:X\to Y$ are form-continuous. By \ref{ContExtOfFormContThm} every uniformly continuous function $f:X\to Y$
extends over $X\cup \partial(\omega(\mathcal{U}))$
and by \ref{IffExtendsItIsFormContinuous} only uniformly continuous functions $f:X\to Y$ extend over $X\cup \partial(X,\omega(\mathcal{U}))$. That's the characteristic property of Samuel-Smirnov compactification of $(X,\mathcal{U})$.
\end{proof}

\section{Forms vs large scale structures}\label{Forms vs large scale structures}
This section is devoted to interaction between forms and large scale structures. Namely, every form $\omega$ on $X$ induces a large scale structure $\mathcal{LS}(\omega)$ and every bornological large scale structure $\mathcal{LS}$ on $X$ induces a form $\omega(\mathcal{LS})$ on $X$. In case of a metrizable large structure $L$, its form carries all the information about $L$ (see \ref{FormsForMetrizableStructures} and \label{LSOfOmegaIsLSForMetric}).

For basic facts related to the coarse category see \cite{Roe lectures}.

Recall that a \textbf{coarse structure} $\mathcal{C}$ on $X$ is a family of subsets $E$
(called \textbf{controlled sets})
of $X\times X$ satisfying the following properties:
\begin{enumerate}
\item The diagonal $\Delta=\{(x,x)\}_{x\in X}$ belongs to $\mathcal{C}$.
\item $E_1\in\mathcal{C}$ implies $E_2\in\mathcal{C}$ for every $E_2\subset E_1$.
\item $E\in\mathcal{C}$ implies $E^{-1}\in\mathcal{C}$, where $E^{-1}=\{(y,x)\}_{(x,y)\in E}$.
\item $E_1,E_2\in\mathcal{C}$ implies $E_1\cup E_2\in\mathcal{C}$.
\item $E, F\in\mathcal{C}$ implies $E\circ F\in\mathcal{C}$, where $E\circ F$
consists of $(x,y)$ such that there is $z\in X$ so that $(x,z)\in E$
and $(z,y)\in F$.
\end{enumerate}

\begin{Definition}\label{LSStructureDef} \cite{DH} 
A \textbf{large scale structure} $\mathcal{LSS}_X$ on a set $X$ is a non-empty set of families $\mathcal{B}$
of subsets of $X$ (called {\it uniformly $\mathcal{LSS}_X$-bounded} or \textbf{uniformly bounded} once $\mathcal{LSS}_X$ is fixed)
satisfying the following conditions:
\begin{enumerate}
\item $\mathcal{B}_1\in\mathcal{LSS}_X$ implies $\mathcal{B}_2\in\mathcal{LSS}_X$ if each element of $\mathcal{B}_2$
consisting of more than one point
is contained in some element of $\mathcal{B}_1$.
\item $\mathcal{B}_1,\mathcal{B}_2\in\mathcal{LSS}_X$ implies $st(\mathcal{B}_1,\mathcal{B}_2)\in\mathcal{LSS}_X$.
\end{enumerate}

A subset $B$ of $X$ is \textbf{bounded with respect to} $\mathcal{LSS}_X$
if the family $\{B\}$ belongs to $\mathcal{LSS}_X$. Thus, each large scale structure on $X$ induces a bornology on $X$ if every finite subset of $X$ is bounded.
\end{Definition}

Roe uses the term coarsely connected if if every finite subset of $X$ is bounded. We find it more appropriate to use the following terminology:

\begin{Definition}
A \textbf{bornological large scale space} is a large scale space $X$ in which if every finite subset of $X$ is bounded. Thus, bounded subsets of $X$ do form a bornology.
\end{Definition}

As described in \cite{DH}, the transition between the two structures is as follows:\\
1. Given a uniformly bounded family $\mathcal{U}$ in $X$, the set $\bigcup\limits_{B\in \mathcal{U}}B\times B$ is a controlled set,\\
2. Given a controlled set $E$, the family $\{E[x]\}_{x\in X}$ is uniformly bounded,
where $E[x]:=\{y\in X\mid (x,y)\in E\}$.

\begin{Definition}\label{OmegaOnLSDefinition}
Given a bornological large scale space $(X,\mathcal{L})$, the form $\omega(\mathcal{L})$ is defined as follows:\\
$$\omega(\mathcal{L})(C_1,\ldots,C_k)=0 \iff
\bigcap\limits_{i=1}^k st(C_i,\mathcal{V})\mbox{ is bounded}$$
for each uniformly bounded cover $\mathcal{V}$ of $X$.
\end{Definition}

\begin{Proposition}\label{FormsForMetrizableStructures}
If $\omega(\mathcal{L}_1)=\omega(\mathcal{L}_2)$ for two metrizable large scale structures on $X$, then $\mathcal{L}_1=\mathcal{L}_2$.
\end{Proposition}
\begin{proof}
Suppose $\mathcal{U}$ is uniformly bounded with respect to $\mathcal{L}_1$ induced by a metric $d_1$ but is not is uniformly bounded with respect to $\mathcal{L}_2$ induced by a metric $d_2$.
That means existence of two sequences $C:=\{x_n\}_{n=1}^\infty$, $D:=\{y_n\}_{n=1}^\infty$ and a number $r > 0$ such that $d_1(x_n,y_n) < r$ for each $n\ge 1$ but $d_2(x_n,y_n)\to\infty$. Notice
$\omega(\mathcal{L}_1)(C,D)=\infty$ but $\omega(\mathcal{L}_2)(C,D)=0$, a contradiction.
\end{proof}

\begin{Proposition}\label{LSOfOmegaPropDef}
Given a form $\omega$ on $X$, the family of all covers $\mathcal{U}$ of $X$ with the property
$$\omega(C_1,\ldots,C_k)=\omega(st(C_1,\mathcal{U}),\ldots,st(C_k,\mathcal{U}))$$
for all vectors $(C_1,\ldots,C_k)$ in $X$ induces a large scale structure $LS(\omega)$ whose family of bounded subsets contains all $\omega$-bounded sets. If $B$ is bounded with respect to $LS(\omega)$ and $\omega(B)=\infty$,
then $\omega(C)=\infty$ for all $C$ intersecting $B$. In particular, if every point in $X$ is $\omega$-bounded, then $LS(\omega)$ is a bornological large scale structure with the same bornology as $\omega$.
\end{Proposition}
\begin{proof}
Notice one can define $LS(\omega)$ as consisting of all covers $\mathcal{U}$
satisfying
$$\omega(C_1,\ldots,C_k)\ge \omega(st(C_1,\mathcal{U}),C_2,\ldots,C_k)$$
for all vectors $(C_1,\ldots,C_k)$ in $X$.
Suppose $\mathcal{B}_1\in LS(\omega)$ and each element of a cover $\mathcal{B}_2$ of $X$
consisting of more than one point
is contained in some element of $\mathcal{B}_1$. For all vectors $(C_1,\ldots,C_k)$ in $X$,
$\omega(C_1,\ldots,C_k)\leq\omega(st(C_1,\mathcal{B}_2),\ldots,st(C_k,\mathcal{B}_2))\leq \omega(st(C_1,\mathcal{B}_1),\ldots,st(C_k,\mathcal{B}_1))=\omega(C_1,\ldots,C_k)$.
$\mathcal{B}_1,\mathcal{B}_2\in LS(\omega)$ implies $st(\mathcal{B}_1,\mathcal{B}_2)\in LS(\omega)$ follows from the fact
$st(C,st(\mathcal{B}_1,\mathcal{B}_2))\subset st(st(C,\mathcal{B}_1),\mathcal{B}_2))\cup 
st(st(C,\mathcal{B}_2),\mathcal{B}_1)$. 
Indeed, for all vectors $(C_1,\ldots,C_k)$ in $X$,
$$\omega(st(C_1,st(\mathcal{B}_1,\mathcal{B}_2)),C_2,\ldots,C_k)\leq$$
$$\omega(st(st(C_1,\mathcal{B}_1),\mathcal{B}_2)),C_2,\ldots,C_k)+
\omega(st(st(C_1,\mathcal{B}_2),\mathcal{B}_1)),C_2,\ldots,C_k)=$$
$$\omega(C_1,C_2,\ldots,C_k)+\omega(C_1,C_2,\ldots,C_k)=\omega(C_1,C_2,\ldots,C_k).$$
Thus, $LS(\omega)$ is indeed a large scale structure.

If $\omega(B)=0$ and $\mathcal{U}=\{B\}\cup\bigcup\limits_{x\in X}\{x\}$,
then for all vectors $(C_1,\ldots,C_k)$ in $X$,
$\omega(C_1,\ldots,C_k)\leq\omega(st(C_1,\mathcal{U}),\ldots,st(C_k,\mathcal{U}))\leq \omega(C_1\cup B,\ldots,C_k\cup B)=\omega(C_1,\ldots,C_k)$. That means $B$ is a bounded subset of $X$ with respect to $LS(\omega)$. Suppose $B\ne\emptyset$ is a bounded subset of $X$ with respect to $LS(\omega)$ and $\omega(B)=\infty$.
Let
$\mathcal{U}:=\{B\}\cup\bigcup\limits_{x\in X}\{x\}$
and suppose $C$ intersects $B$. Now, $\omega(C)=\omega(st(C,\mathcal{U}))=\omega(C\cup B)=\infty$.
\end{proof}

\begin{Proposition}\label{CharOfLSInducedByNormalForms}
If $\omega$ is a normal $T_1$ form, then the following conditions are equivalent:\\
a. $\mathcal{U}\in LS(\omega)$,\\
b. For each neighborhood $W$ of $\mathcal{P}\in\partial(\omega)$
in $X\cup\partial(\omega)$
there is a neighborhood $W'$ of $\mathcal{P}$ such that
$st(W',\mathcal{U})\subset W$,\\
c. The coronas of $C$ and $st(C,\mathcal{U})$ coincide for each $C\subset X$.
\end{Proposition}
\begin{proof}
a)$\implies$b). Since $X\cup\partial(\omega)$ is compact Hausdorff,
there is a neighborhood $U$ of $\mathcal{P}$ in $X\cup\partial(\omega)$
whose closure is contained in $W$. In particular, $\omega(U\cap X,X\setminus W)=0$. Now, $\omega(st(U\cap X,\mathcal{U}),st(X\setminus W,\mathcal{U}))=0$ and $B:=st(U\cap X,\mathcal{U})\cap st(X\setminus W,\mathcal{U})$ is $\omega$-bounded. Put $W':=U\setminus st(B,\mathcal{U})$.

b)$\implies$c). The corona of $D$ is defined as all $\mathcal{P}\in \partial(\omega)$ contained in $cl(D)$. Equivalently, all $\mathcal{P}$ containing $D$.

Suppose $\mathcal{P}$ contains $st(C,\mathcal{U})$ but not $C$.
There exists a vector $(C_1,\ldots,C_k)$ consisting of elements of $\mathcal{P}$ such that
$\omega(C,C_1,\ldots,C_k)=0$. Consequently,
$$\omega(st(C,\mathcal{U}),st(C_1,\mathcal{U}),\ldots,st(C_k,\mathcal{U}))=0,$$
 a contradiction as $\mathcal{P}$ contains all those sets.

c)$\implies$a). Suppose $U\notin LS(\omega)$.
There is a vector $(C_1,\ldots,C_k)$ satisfying
$\omega(C_1,\ldots,C_k)=0$ and $\omega(st(C_1,\mathcal{U}),\ldots,st(C_k,\mathcal{U}))=\infty$.
Choose sets $D_1,\ldots,  D_k$ such that $\omega(C_i,D_i)=0$ for each $i\leq k$ and $\bigcup\limits_{i=1}^k D_i=X$.
Thus, coronas of $D_i$ and $C_i$ are disjoint for each $i$ resulting
in $\omega(st(C_i,\mathcal{U}),st(D_i,\mathcal{U}))=0$ for each $i\leq k$.
That contradicts $\omega(st(C_1,\mathcal{U}),\ldots,st(C_k,\mathcal{U}))=\infty$.
\end{proof}

\begin{Proposition}\label{OmegaOfMaxBDDStructure}
Given an infinite set $X$ consider the maximal bounded geometry large scale structure $\mathcal{LS}$ on $X$. It consists of all covers $\mathcal{U}$ of $X$ satisfying the following properties:\\
a. There is a natural number $n\ge 1$ such that each element $U$ of $\mathcal{U}$ has at most $n$ points.\\
b. There is a natural number $m\ge 1$ such that each point $x$ of $X$ belongs to at most $m$ elements $U$ of $\mathcal{U}$.\\
The form $\omega(\mathcal{LS})$ induced by $\mathcal{LS}$ is characterized by the following property: $\omega(\mathcal{LS})(V)=0$ if and only if at least one coordinate of $V$ is finite.
\end{Proposition}
\begin{proof}
Suppose all $C_i$ are infinite and $\omega(\mathcal{LS})(C_1,\ldots,C_k)=0$.
We may assume $k$ is the smallest number with that property. Then we can maximize the number of pairs $(C_i,C_j)$ such that $C_i\cap C_j=\emptyset$. We claim all such pairs are disjoint if $i\ne j$. Indeed, assume
$C_1\cap C_2\ne\emptyset$ by reordering, if necessary.
If $C_1\cap C_2$ is finite, we can replace each $C_i$ by $C_i\setminus C_1\cap C_2$ and increase the number of disjoint pairs. If $C_1\cap C_2$ is infinite, we can remove $C_1$ and replace $C_2$ by $C_1\cap C_2$.
Finally, we can pick an infinite countable subset $D_i$ of $C_i$ for each $i\leq k$. Pick bijections $f_i:D_1\to D_i$
and consider the trivial extension $\mathcal{U}$ of the family $U_x:=\{x,f_2(x),\ldots,f_k(x)\}$, $x\in D_1$. Notice that $D_1\subset st(D_i,\mathcal{U})$ for all $i\leq k$, a contradiction.
\end{proof}

\begin{Proposition}\label{LSOfMinOmega}
Suppose $X$ is an infinite set.
Let $\omega$ be defined as follows: $\omega(V)=0$ if and only if at least one coordinate of $V$ is finite.
The large scale structure $LS(\omega)$ induced by $\omega$ consists of all covers $\mathcal{U}$ of $X$ with the property $st(F,\mathcal{U})$ is finite for all finite subsets $F$ of $X$.
\end{Proposition}
\begin{proof}
Suppose $\mathcal{U}\in LS(\omega)$ and $F\subset X$ is finite.
Now, $0=\omega(F)=\omega(st(F,\mathcal{U}))$, so $st(F,\mathcal{U})$ is finite.

The remainder of the proof is obvious.
\end{proof}

\begin{Question}
Is $LS(\omega)$ normal if $\omega$ is the form defined in \ref{LSOfMinOmega}?
\end{Question}

\begin{Proposition}\label{FormInequality}
If $\lambda$ is a normal $T_1$ form on $X$, then $\omega(\mathcal{LS}(\lambda))\leq \lambda$.
\end{Proposition}
\begin{proof}
See \ref{OmegaOnLSDefinition} for the definition of the form induced by a large scale structure. If $\lambda(C_1,\ldots,C_k)=0$, then for any
$\mathcal{U}\in \mathcal{LS}(\lambda)$ one has
$\lambda(st(C_1,\mathcal{U}),\ldots,st(C_k,\mathcal{U}))=0$.
In particular, $\bigcap\limits_{i=1}^k st(C_i,\mathcal{U})$ must be
$\mathcal{LS}(\lambda)$-bounded (see \ref{LSOfOmegaPropDef}), hence\\
$\omega(\mathcal{LS}(\lambda))(C_1,\ldots,C_k)=0$.
\end{proof}

\begin{Question}
Suppose $\lambda$ is a normal $T_1$ form on a set $X$. Does $\omega(\mathcal{LS}(\lambda))=\lambda$ hold?
\end{Question}

The answer to the above question is yes in case of forms induced by metrics:
\begin{Proposition}\label{LSOfOmegaIsLSForMetric}
Given a metric space $(X,d)$, the basic large scale form $\omega_l(X,d)$ satisfies $LS(\omega_l(X,d))=LS(X,d)$.
\end{Proposition}
\begin{proof}
Put $\lambda=\omega_l(X,d)$ (see \ref{BasicLargeScaleFormOnMetric}). Observe that $LS(\lambda)$ consists of all covers $\mathcal{U}$ of $X$ that have a finite upper bound on diameters.
Indeed, if $U_n\in \mathcal{U}$ and $diam(U_n)\to \infty$, then we can find points $x_n,y_n\in U_n$ such that $d(x_n,y_n)\to\infty$.
In that case $\lambda(X_1,Y_1)=0$, where $X_1=\{x_n\}_{n=1}^\infty$
and $Y_1=\{y_n\}_{n=1}^\infty$. However, $\lambda(st(X_1,\mathcal{U}),st(Y_1,\mathcal{U}))=\infty$.
\end{proof}

\begin{Observation}
The large scale structure $\mathcal{LS}$ described in \ref{OmegaOfMaxBDDStructure} has the property that 
$LS(\omega(\mathcal{LS}))\ne \mathcal{LS}$.
\end{Observation}

\begin{Problem}
Characterize normality of $LS(\omega)$ in terms of slowly oscillating functions.
\end{Problem}

\begin{Question}
Is $LS(\omega)$ normal if $\omega$ is normal?
\end{Question}

\begin{Question}
Is $\omega(\mathcal{LS})$ normal if $\mathcal{LS}$ is normal?
\end{Question}

\section{Higson corona of large scale forms}
Given a large scale normal formed set $(X,\omega)$ we show that its boundary $\partial(\omega)$ is an analog of Higson corona if it is Hausdorff. Notice that Roe \cite{Roe lectures} (p.29) Defines Higson coronas only for proper coarse spaces. We do not need that restriction.

Recall $\omega$ is a large scale form if $\omega(x,X)=0$ for each $x\in X$. Thus, the topology induced by $\omega$ on $X$ is discrete.

Consider the complex numbers $\mathbb{C}$ with the basic topological form.

\begin{Definition}\label{HigsonFunctionDef}
A \textbf{Higson function} on $(X,\omega)$ is a form-continuous function
$f:X\to \mathbb{C}$ with bounded image.
The $C^\ast$-algebra of Higson functions on $(X,\omega)$ is denoted by $C_h(\omega)$ (use \ref{ProductOfPerpFunctions} to see that Higson functions do form a $C^\ast$-algebra).
$C_0(\omega)$ is the subalgebra of $C_h(\omega)$ consisting of Higson functions vanishing at infinity, i.e. for each $\epsilon > 0$
the set $\{x\in X | |f(x)| > \epsilon\}$ is bounded.
\end{Definition}

\begin{Proposition}\label{HigsonCoronaProp}
Given a large scale normal formed set $(X,\omega)$, the $C^\ast$-algebra $C(\partial(\omega))$ of continuous functions
on $\partial(\omega)$
is isomorphic to $C_h(\omega)/C_0(\omega)$.
\end{Proposition}
\begin{proof}
Given a continuous function $f:\partial(\omega)\to [-r,r]\times [-r,r]$ we can extend it over $X\cup\partial(\omega)$ due to topological normality of $X\cup\partial(\omega)$.
Any such extension is form-continuous by
\ref{IffExtendsItIsFormContinuous}. The difference $f_1-f_2$ of two extensions equals $0$ on $\partial(\omega)$, hence,
for each $\epsilon > 0$,
the set $\{x\in X | |f(x)| > \epsilon\}$ is bounded due to $X\cup\partial(\omega)$ being large scale compact.

Conversely, any form-continuous $f:X\to [-r,r]\times [-r,r]$ extends to a continuous function on $X\cup\partial(\omega)$ by
\ref{ContExtOfFormContThm}.
\end{proof}

\begin{Theorem}\label{AsdimVsDimOfHigsonCorona}
Suppose $(X,\mathcal{LS})$ is a large scale space of finite asymptotic dimension $asdim(X,\mathcal{LS})$. If $(X,\mathcal{LS})$
is metrizable and $\omega$ is the form induced by $\mathcal{LS}$, then $asdim(X,\mathcal{LS})=\dim(\partial(\omega))$.
\end{Theorem}
\begin{proof}
Pick a discrete metric $d$ on $X$ inducing $\mathcal{LS}$.
Let $n=\dim(\partial(\omega))$. According to \cite{DM},
$asdim(X,d)$ is the smallest integer $k\ge -1$ such that
every slowly oscillating function $f:A\subset X\to S^k$ extends over $X$ to a slowly oscillating function. If $k=n$ it is so. Indeed,
$f$ extends over $A\cup \partial(\omega_A)$ by \ref{ContExtOfFormContThm}, then over $X\cup\partial(\omega_A)$
due to $\dim(\partial(\omega))=n$. Next, it extends over a closed naighborhood $D$ of $X\cup\partial(\omega_A)$. Since the complement of $int(D)$ is discrete, we can extend over the whole
$X\cup\partial(\omega)$. Notice the restriction of the last extension to $X$ is a slowly oscillating extension of $f$.

If $k < n$, then there is a continuous function $g:C\to S^{k-1}$, $C$ a closed subset of $\partial(\omega)$ that does not extend over
$\partial(\omega)$. We can extend $g$ over a closed neighborhood
$A$ of $C$ in $X\cup\partial(\omega)$. The restriction $h$ of that extension to $A\cap X$ is slowly oscillating but it does not extend over $X$. Inddeed, any such extension would generate an extension of $g$ over $\partial(\omega)$.
\end{proof}

\section{Compatible forms}
In practical applications, any large scale formed set $(X,\omega_l)$
comes with a topology $\mathcal{T}$ and we want to change the topology
on $X\cup\partial(\omega_l)$ so that the new topology is compact Hausdorff, is the same on $\partial(\omega_l)$, and equals $\mathcal{T}$ when restricted to $X$. One way to achieve it is to consider a small scale form $\omega_{ss}$ on $X$ inducing $\mathcal{T}$
and look at conditions when $X\cup\partial(\omega_l+\omega_{ss})$
is the space combining both topologies.

The following is obvious.
\begin{Proposition}
Given two forms $\omega_1$ and $\omega_2$ on a set $X$,
$\omega_1+\omega_{2}$ is a form.
\end{Proposition}

\begin{Proposition}
Given two forms $\omega_1$ and $\omega_2$ on a set $X$,
the topology $T(\omega)$ induced by
$\omega=\omega_1+\omega_{2}$ equals $T(\omega_2)$ if $\omega_1$ is large scale.
\end{Proposition}
\begin{proof}
Since $\omega_1(x,X)=0$ for all $x\in X$,
$\omega(x,X\setminus U)=0$ is equivalent to $\omega_2(x,X\setminus U)=0$. Thus, $T(\omega)=T(\omega_2)$.
\end{proof}

\begin{Definition}\label{CompatibilityDef}
Suppose $(X,\omega_l)$ is a large scale normal formed set
and $(X,\omega_{ss})$ is the small scale normal formed set
induced by a normal topology $\mathcal{T}$ on $X$.
$\omega_l$ and $\omega_{ss}$ are \textbf{compatible} if the following conditions are satisfied:\\
1. $\omega:=\omega_l+\omega_{ss}$ is normal and $T_1$.\\
2. The identity function $i:X\to X$ induces
a continuous extension $\tilde i:X\cup\partial(\omega_l)\to X\cup\partial(\omega_l+\omega_{ss})$ that is a homeomorphism from
$\partial(\omega_l)$ onto $\partial(\omega_l+\omega_{ss})\setminus X$.
\end{Definition}

\begin{Observation}
Compatibility of $\omega_l$ and $\omega_{ss}$ implies that
$(X,\mathcal{T})$ is locally compact as it is an open subset of 
a compact Hausdorff space $\partial(\omega_l+\omega_{ss})$.
\end{Observation}

\begin{Theorem}\label{TopologicalNormalityCriterion}
Suppose $(X,\omega_l)$ is a large scale normal formed set
and $(X,\omega_{ss})$ is the small scale normal formed set
induced by a normal topology $\mathcal{T}$ on $X$.
$\omega:=\omega_l+\omega_{ss}$ is normal and $T_1$ if the following conditions are satisfied:\\
1. For every $\omega_l$-bounded
subset $B$ of $X$ there is an $\omega_{ss}$-open subset $U$ containing $B$ and being $\omega_l$-bounded.\\
2. $\omega_l(C_1,\ldots,C_k)=0$ 
implies $\omega_l(cl(C_1),\ldots,C_k)=0$, where the closure is taken with respect to $\mathcal{T}$.
\end{Theorem}
\begin{proof}
Suppose $x\ne y\in X$. Since $\omega_l$ is large scale, both $x$ and $y$ are $\omega_l$-bounded, so $\omega_l(x,y)=0$.
Since $\mathcal{T}$ is a normal topology, $\omega_{ss}(x,y)=0$
resulting in $\omega(x,y)=0$. Thus, $\omega$ is $T_1$.

Suppose $\omega(C_1,\ldots,C_k)=0$. Replacing each $C_i$ by $cl(C_i)$ we may assume each $C_i$ is $\mathcal{T}$-closed.
Now, $\omega_l(C_1,\ldots,C_k)=0$ and $\bigcap\limits_{i=1}^k C_i=\emptyset$. Due to normality of $\omega_l$ there exist
sets $D_1,\ldots, D_k$ forming a cover of $X$ such that
$\omega_l(C_i,D_i)=0$ for each $i\ge 1$. Pick a $\mathcal{T}$-open set $U_i$
containing $C_i\cap D_i$ such that $cl(U_i)$ is $\omega_l$-bounded. Let $Y:=\bigcup\limits_{i=1}^k cl(U_i)$.
It is a normal space, so (see \ref{ZeroSetsLemma})
there exist $\mathcal{T}$-closed sets $E_i\subset Y$ such that $E_i\cap C_i=\emptyset$ for each $i\leq k$, $Y=\bigcup\limits_{i=1}^k E_i$,
and $\bigcap\limits_{i=1}^k E_i=\emptyset$.
Let $D'_i:=D_i\cup E_i$ for each $i\leq k$. Observe that
$X=\bigcup\limits_{i=1}^k D'_i$, $\omega_l(C_i,D'_i)=0$
and $\omega_{ss}(C_i,D'_i)=0$ for each $i\leq k$.
Therefore $\omega(C_i,D'_i)=0$ for each $i\leq k$ and $\omega_l+\omega_{ss}$ is normal.
\end{proof}

\begin{Theorem}\label{TopologicalCompatibilityCriterion}
Suppose $(X,\omega_l)$ is a large scale normal formed set, $(X,\omega_{ss})$ is a small scale normal formed set
induced by a normal topology $\mathcal{T}$ on $X$, and each $\omega_l$-bounded set is pre-compact in $\mathcal{T}$.
$\omega_l$ and $\omega_{ss}$ are compatible if the following conditions are satisfied:\\
1. For every $\omega_l$-bounded
subset $B$ of $X$ there is an $\omega_{ss}$-open subset $U$ containing $B$ and being $\omega_l$-bounded.\\
2. $\omega_l(C_1,\ldots,C_k)=0$ 
implies $\omega_l(cl(C_1),\ldots,C_k)=0$, where the closure is taken with respect to $\mathcal{T}$.
\end{Theorem}
\begin{proof}
Put $\omega:= \partial(\omega_l+\omega_{ss})$.
Since $i:(X,\omega_l)\to (X,\omega)$ is form-continuous, it extends to a continuous map $\tilde i:\tilde i:X\cup\partial(\omega_l)\to X\cup\partial(\omega)$. Given $\mathcal{P}\in \partial(\omega_l)$,
$\tilde i(\mathcal{P})$ consists of all $C\subset X$ such that $\omega(C,D)=\infty$ for all $D\in \mathcal{P}$ (see \ref{ContExtOfFormContThm}).
Suppose $C\in \tilde i(\mathcal{P})\setminus \mathcal{P}$. 
That means existence of $E\in \mathcal{P}$ such that $\omega_l(C,E)=0$.
Hence, $cl(C)\cap cl(E)$ is $\omega_l$-bounded and is contained in an $\omega_l$-bounded $U\in \mathcal{T}$. Now, $E\setminus U\in P$, $\omega_{ss}(C,E\setminus U)=0$ and $\omega_l(C,E\setminus U)=0$ resulting in $\omega(C,E\setminus U)=0$, a contradiction.

So far we have established $\tilde i$ is an embedding on $\partial(\omega_l)$ and the remaining task is to show $\tilde i(\partial(\omega_l))=\partial(\omega)\setminus X$.
Suppose $\mathcal{Q}\in (\partial(\omega)\setminus X)\setminus \partial(\omega_l)$. That means existence of $C_1,\ldots,C_k\in \mathcal{Q}$
satisfying $\omega_l(C_1,\ldots,C_k)=0$. Notice $B:=\bigcap\limits_{i=1}^k cl(C_i)$ is compact in $\mathcal{T}$ as it is $\omega_l$-bounded. Since $X\cup\partial(\omega)$ is compact Hausdorff, there is a neighborhood $o(U)$ of $\mathcal{Q}$ whose closure is disjoint with $B$. Now,
$\omega_l(U,C_1,\ldots,C_k)=0$ and $\omega_{ss}(U,C_1,\ldots,C_k)=0$, so $\omega(U,C_1,\ldots,C_k)=0$ which contradicts
$U\in \mathcal{Q}$.
\end{proof}

\begin{Corollary}
If $(X,d)$ is a proper metric space, then $X\cup\partial(\omega)$ is the Higson compactification of $X$, where $\omega$ is the sum of the large scale form $\omega_l$ of $(X,d)$ and the basic topological form $\omega_{ss}$ of $(X,d)$.
\end{Corollary}
\begin{proof}
$(X,d)$ being proper means every $d$-bounded set is pre-compact in the topology $\mathcal{T}$ induced by $d$.
Therefore, every $\omega_l$-bounded subset of $X$ is pre-compact in $\mathcal{T}$. Also, for each $r > 0$ one has
$B(C,r)=B(cl(C),r)$ which implies $\omega_l$ and $\omega_{ss}$ are compatible.

To show $X\cup\partial(\omega)$ is the Higson compactification of $X$ we need to show that every continuous, bounded, and slowly oscillating $f:X\to \mathbb{C}$ extends over $X\cup\partial(\omega)$ and every continuous and bounded function $f:X\to \mathbb{C}$ that does admit continuous extension is slowly oscillating (see \cite{Roe lectures}).

Suppose $f:X\to \mathbb{C}$ is continuous, bounded, and slowly oscillating. By \ref{SlowlyOscillatingExampleOfPerpContinuous}
and \ref{FormContinuousVsPerpContinuous}, $f:X\to cl(f(X))$ is form
continuous, so it extends over $X\cup\partial(\omega)$
to a continuous function by \ref{ContExtOfFormContThm}.
Conversely, any continuous bounded function $f:X\to \mathbb{C}$
that extends over $X\cup\partial(\omega)$ to a continuous function
must be form-continuous with respect to $\omega_l$, hence slowly oscillating. 
\end{proof}

\subsection{Gromov compactification}
In this part we consider proper geodesic $\delta$-hyperbolic spaces $(X,d)$ that are \textbf{visual}, i.e. there is $p\in X$ such that the union of all geodesic rays emanating from $p$ equals $X$.
If $\omega_{ss}$ is the basic topological form of $(X,d)$ and $\omega_h$ is the hyperbolic form of $X$, then $\omega_h$ and $\omega_{ss}$ are compatible and $X\cup \partial(\omega_h+\omega_{ss})$ is the well-known Gromov compactification of $(X,d)$.

\begin{Proposition}
If $\omega_{ss}$ is the basic topological form of $(X,d)$ and $\omega_h$ is the hyperbolic form of $X$, then $\omega_h$ and $\omega_{ss}$ are compatible.
\end{Proposition}
\begin{proof}
By \ref{CharHypPerpendicularity}, $C\perp X$ means $C\subset X\setminus N(X,r)$ for some $r > 0$. Notice $N(X,r)\supset X\setminus B(p,2r)$, so $C$ is $\omega_h$-bounded if and only if $C$ is $d$-bounded, so $C$ clearly can be enlarged to an open $d$-bounded set.
Also, $N(cl(C),r)=N(C,r)$ for each $r > 0$.
\end{proof}

\begin{Lemma}\label{UnionOfNOfRays}
For each $r > 0$ there are geodesic rays $l_1,\ldots, l_k$ emanating from $p$ such that
$$X=B(p,r+\delta)\cup\bigcup\limits_{i=1}^k N(l_i,r).$$
\end{Lemma}
\begin{proof}
On the $r$-sphere $S$ centered at $p$ choose points $x_1,\ldots,x_k$ such that $S=\bigcup\limits_{i=1}^k B(x_i,\delta)$.
For each $i\leq k$ choose a geodesic ray $l_i$ emanating from $p$ and containing $x_i$. Suppose $d(p,x)\ge r+\delta$.
Choose a geodesic ray $l$ emanating from $p$ and containing $x$. 
Let $y\in S\cap l$ and pick $j\leq k$ so that $d(y,x_j) < \delta$.
Now, $2\left<x,x_j\right>_p=d(p,x)+r+\delta-d(x,x_j) > 2r$,
so $x\in N(l_j,r)$.
\end{proof}

\begin{Corollary}\label{HypPointsAtInfinityAreRays}
For each $\mathcal{P}\in \partial(\omega_h)$ there is a geodesic ray $l$ emanating from $p$ such that $l\in \mathcal{P}$.
\end{Corollary}
\begin{proof}
For each $n\ge 1$ there is a geodesic ray $l_n$ emanating from $p$ such that $N(l_n,n)\in P$ (use \ref{UnionOfNOfRays}).
Let $\alpha_n:[0,\infty)\to l_n$ be a parametrization of $l_n$ for each $n\ge 1$. By a diagonal argument choose an increasing sequence $n(k)$ with the property $\alpha_{n(k)}(q)$ is convergent
to $\alpha(q)$ for each rational $q$. Notice $\alpha$ extends to a parametrization $\alpha$ of a ray $l$ and $\alpha_{n(k)}(t)$ is convergent
to $\alpha(t)$ for each $t\ge 0$.
Il $l\notin P$, then there is $C\in P$, $C\perp l$.
Therefore $C\perp N(l,r)$ for some $r > 0$ (see \ref{CharHypPerp2}).
Find $k\ge 1$ such that for $m=n(k) > r+\delta$, $d(\alpha_{m}(r+\delta),\alpha(r+\delta)) < \delta$. Suppose $x\in N(l_m,m)$
and choose $y\in l_m$ such that $\left<x,y\right>_p > m$.
Put $z:=\alpha(r+\delta)$, $t:=\alpha_{m}(r+\delta)$, and notice $\left<y,z\right>_p 
=0.5\cdot (d(y,t)+r+\delta+r+\delta-d(y,z)) >
0.5\cdot (d(y,t)+r+\delta+r+\delta-(d(y,t)+\delta))=
0.5\cdot (2r+\delta)=r+\delta/2$.
Therefore, $ \left<x,z\right>_p \ge \min(\left<x,y\right>_p,\left<z,y\right>_p)-\delta/4 > \min(m,r+\delta/2)-\delta/4 > r$.
Consequently, $N(l_m,m)\subset N(l,r)$ and $C\perp N(l_m,m)$
which contradicts $N(l_m,m)\in P$.
\end{proof}

\begin{Proposition}
The relation $l_1\sim l_2$ defined as $\omega_h(l_1,l_2)=\infty$
among geodesic rays emanating from $p$ is an equivalence relation
and each $\mathcal{P}\in\partial(\omega_h)$ contains exactly one equivalence class.
\end{Proposition}
\begin{proof}
Pick parametrizations $\alpha_i:[0,\infty)\to l_i$ for $i\leq 2$.
Notice that $\left<\alpha_1(t),\alpha_2(t)\right>_p$
is an increasing function of $t$ and it approaches $\infty$ if and only if $\omega_h(l_1,l_2)=\infty$. Given a third ray $l_3$
with parametrization $\alpha_3$ and $\omega_h(l_2,l_3)=\infty$,
the inequality $ \left<\alpha_1(t),\alpha_3(t)\right>_p\ge
\min(\left<\alpha_1(t),\alpha_2(t)\right>_p,\left<\alpha_2(t),\alpha_3(t)\right>_p-\delta/4$ shows $\omega_h(l_1,l_3)=\infty$.

If $\omega_h(l_1,l_2)=\infty$, then for each $r > 0$
there is $M > 0$ such that $\left<\alpha_1(t),\alpha_2(t)\right>_p > r+\delta$. If $ \left<\alpha_1(t),x\right>_p > r+M+\delta$,
then $t > M+\delta$ and \\
$ \left<x,\alpha_2(M)\right>_p\ge
\min(\left<\alpha_1(t),x\right>_p,\left<\alpha_1(t),\alpha_2(t)\right>_p)-\delta/4 > r$.\\
Thus, $N(l_1,r+M+\delta)\subset N(l_2,r)$.
In particular, $\omega_h(l_1,C_1,\ldots,C_k)=\infty$ implies
$\omega_h(l_1,l_2,C_1,\ldots,C_k)=\infty$, so any $P$ containing
$l_1$ contains $l_2$ as well.
\end{proof}

\begin{Observation}
One can show existence of a universal constant $K$ depending on $\delta$ ony such that
given parametrizations $\alpha_i:[0,\infty)\to l_i$ for $i\leq 2$
of two equivalent geodesic rays $l_1$ and $l_2$, one has
$d(\alpha_1(t),\alpha_2(t)) < K$ for all $t$ (see Lemma 9.51 on p.229 in \cite{DK}) which is obviously stronger than $l_1\sim l_2$.
\end{Observation}

\begin{Proposition}
Given a geodesic ray $l$ emanating from $p$, the sets $o(N(l,n))$, $n\ge 1$,
form a basis of open sets of $\mathcal{P}\in \partial(\omega_h+\omega_{ss})$
containing $l$.
\end{Proposition}
\begin{proof}
Let $\omega:=\omega_h+\omega_{ss}$.
Since $N(l,n)$ is $d$-open in $X$, it is $\omega$-open.
If $\mathcal{P}\in o(U)$, $U$ being $\omega$-open, then $X\setminus U\notin \mathcal{P}$,
so $(X\setminus U)\cap N(l,k)=\emptyset$ for some $k\ge 1$
and $N(l,k)\subset U$. Since $X\setminus N(l,k)\notin \mathcal{P}$,
$\mathcal{P}\in o(N(l,k))\subset o(U)$.
\end{proof}

\begin{Theorem}\label{CharOfGromovCompactification}
If $\omega_{ss}$ is the basic topological form of a proper visual geodesic $\delta$-hyperbolic space $(X,d)$ and $\omega_h$ is the hyperbolic form of $X$, then $X\cup \partial(\omega_h+\omega_{ss})$ is the compactification of $(X,d)$ characterized by extensions of bounded Gromov functions $X\to R$.
\end{Theorem}
\begin{proof}
A bounded Gromov functions $X\to R$ is a continuous function
such that for each $\epsilon > 0$ there is $r > 0$ so that
$|f(x-f(y)| > \epsilon$ implies $\left<x,y\right>_p < r$.
Notice Gromov functions are identical with bounded form-continuous functions $(X,\omega)\to R$, where $\omega=\omega_h+\omega_{ss}$. Apply \ref{ContExtOfFormContThm} and \ref{IffExtendsItIsFormContinuous}.
\end{proof}

\begin{Remark}
Theorem \ref{CharOfGromovCompactification} shows
that $X\cup \partial(\omega_h+\omega_{ss})$ is the classical Gromov compactification (see 6.4 in \cite{Roe lectures}  on p.93).
\end{Remark}

\section{Dimension of formed sets}

The goal of this section is to give simple proofs of results that generalize the work of Austin-Virk \cite{AV} on Dimension Raising maps in coarse category.

\begin{Definition}
A form-continuous function $f:(X,\omega_X)\to (Y,\omega_Y)$ of formed sets is \textbf{proper} if images of $\omega_X$-bounded sets are $\omega_Y$-bounded.
\end{Definition}

\begin{Theorem}\label{CharOfNTo1OnBd}
Suppose $f:(X,\omega_X)\to (Y,\omega_Y)$ is a surjective and proper form-continuous function of normal $T_1$ formed sets and $n\ge 1$.
The induced map $$\partial f:\partial(\omega_X)\to \partial(\omega_Y)$$
is always surjective and it is $n$-to-$1$ if and only if for each sequence $\{A_i\}_{i=1}^{n+1}$ of subsets of $X$, the condition $\omega_X(A_i,A_j)=0$ for all $i\ne j$ implies
$\omega_Y(f(A_1),\ldots,f(A_{n+1}))=0$.
\end{Theorem}
\begin{proof}
See \ref{ContExtOfFormContThm} for a description of $\partial f$.
Using \ref{ClosedMapProp} observe that the continuous extension
$\tilde f$ of $f$ is closed, hence the extension $\tilde f$ of $f$ is surjective
($Y$ is dense in $Y\cup\partial (\omega_Y)$).
In particular
$\partial f$ is surjective as no point in $X\setminus \partial(\omega_X)$ can be mapped to $\partial (\omega_Y)$.

Suppose $\partial f$ is 
not $n$-to-$1$, i.e. there exists $Q\in \partial(\omega_Y)$ and $\mathcal{Q}_i\in \partial(\omega_X))$, $1\leq i\leq n+1$, such that $(\partial f)(\mathcal{Q}_i)=\mathcal{Q}$ and $\mathcal{Q}_i\ne \mathcal{Q}_j$ if $i\ne j$. Since $X\cup \partial(\omega_X)$ is Hausdorff, there 
exist mutually disjoint closed neighborhoods $D_i$ of $\mathcal{Q}_i$
in $X\cup \partial(\omega_X)$. Put $A_i:=D_i\cap X$.
Notice $A_i\in \mathcal{Q}_i$ for each $i\leq n+1$. Also,
$\omega_X(A_i,A_j)=0$ for all $i\ne j$. However,
$\omega_Y(f(A_1),\ldots,f(A_{n+1}))=\infty$ as $f(A_i)\in \mathcal{Q}$ for each $i\leq n+1$.

Suppose there exist  sets $A_i$, $i\leq n+1$, such that $\omega_X(A_i,A_j)=0$ whenever $i\ne j$ but $\omega_Y(f(A_1),\ldots,f(A_{n+1}))=\infty$.
We may assume each $A_i$ is closed in $X$ by replacing it with $cl(A_i)$
(see \ref{OmegaOnClosuresGeneralCase}).
There is $\mathcal{Q}\in \partial(\omega_Y)$ containing $f(A_i)$ for each $i$. 
In particular, $\mathcal{Q}\in\partial(f(A_i),\omega_Y)$ for each $i$.
Therefore, for each $i\leq n+1$, there is $\mathcal{Q}_i\in\partial(A_i,\omega_X)$
so that $(\partial f)(\mathcal{Q}_i)=\mathcal{Q}$ (use \ref{RelativeTopologyAndForms}). Since $\mathcal{Q}_i\ne \mathcal{Q}_j$ in $\partial(\omega_X)$, we conclude that $\partial f$ is 
not $n$-to-$1$.
\end{proof}

\begin{Theorem}\label{CoarselyNto1Impliesnto1OnCoronas}
Suppose $(X,\mathcal{L}_X)$ and $(Y,\mathcal{L}_Y)$ are large scale spaces such that the induced forms $\omega_X$, $\omega_Y$ are normal and $T_1$.
If $f:(X,\mathcal{L}_X)\to (Y,\mathcal{L}_Y)$ is coarsely $n$-to-$1$ for some $n\ge 1$, then the induced map
$\partial(\omega_X)\to \partial(\omega_Y)$ of form coronas
is $n$-to-$1$.
\end{Theorem}
\begin{proof}
Suppose the induced map
$\partial(\omega_X)\to \partial(\omega_Y)$ of form coronas
is not $n$-to-$1$.
By Theorem \ref{CharOfNTo1OnBd} there exist  sets $A_i\subset X$, $i\leq n+1$ such that $\omega_X(A_i,A_j)=0$ if $i\ne j$ but $\omega_Y(f(A_1),\ldots,f(A_{n+1}))=\infty$. Therefore, there exists a uniformly bounded cover $\mathcal{U}$ of $Y$ such that $\bigcap\limits_{i=1}^{n+1} st(f(A_i),\mathcal{U})$ is $\omega_Y$-unbounded. Choose a uniformly bounded cover $\mathcal{W}$ of $X$
such that each set $f^{-1}(U)$, $U\in st(\mathcal{U},\mathcal{U})$, can be covered by at most $n$-elements of $\mathcal{W}$. Define $C_i$ as $A_i\setminus B_i$, where $B_i:= \bigcup\limits_{j\ne i}st(A_i,\mathcal{W})\cap st(A_j,\mathcal{W})$ is $\omega_X$-bounded. 
If $x\in  \bigcap\limits_{i=1}^{n+1}st(C_i,f^{-1}(\mathcal{U}))$,
then $st(x,f^{-1}(\mathcal{U}))\cap C_i\ne\emptyset$ for each $i$, which implies
existence of $i\ne j$ such that $st(C_i,\mathcal{W})\cap st(C_j,\mathcal{W})\ne\emptyset$, a contradiction. That means $\bigcap\limits_{i=1}^{n+1}st(C_i,f^{-1}(\mathcal{U}))=\emptyset$. Consequently, $\bigcap\limits_{i=1}^{n+1}st(f(C_i),\mathcal{U})=\emptyset$ 
and $\bigcap\limits_{i=1}^{n+1}st(f(A_i),\mathcal{U}))\subset \bigcup\limits_{i=1}^{n+1}st(f(B_i),\mathcal{U})$ is $\omega_Y$-bounded, a contradiction.
\end{proof}

\begin{Theorem}\label{MainCoarseNTo1Thm}
Suppose $f:X\to Y$ is a coarse bornologous function of large scale spaces
whose induced forms are normal and $T_1$. If $X$ is metrizable and $n\ge 1$, then $f$ is coarsely $n$-to-$1$ if and only if the induced map of their boundaries at infinity is $n$-to-$1$.
\end{Theorem}
\begin{proof}
Suppose $f$ is not coarsely $n$-to-$1$ and the induced map of the boundaries at infinity is $n$-to-$1$. That means the existence of a uniformly bounded cover $\mathcal{U}$ of $Y$ such that for each $k\ge 1$ there is
$U_k\in U$ such that $f^{-1}(U_k)$ cannot be covered by at most $n$ sets of diameter at most $2k$. That implies existence, for each $k\ge 1$,
of points $x_1^k,\ldots,x_{n+1}^k$ of $f^{-1}(U_k)$ such that
$d(x_i^k,x_j^k)\ge k$ whenever $i\ne j$. Let $A_i:=\{x_i^k\}_{k=1}^\infty$.
Notice $\omega_X(A_i,A_j)=0$ if $i\ne j$. By Theorem \ref{CharOfNTo1OnBd}, $\omega_Y(f(A_1),\ldots,f(A_{n+1}))=0$.
Let $C_i=st(f(A_i),\mathcal{U})$ for $i\leq n+1$. Therefore $\omega_Y(C_1,\ldots,C_{n+1})=0$. Since $W:=\bigcup\limits_{i=1}^\infty U_n\subset C_i$ for each $i$,
$\omega_Y(W)=0$ and $W$ is bounded in $Y$. Hence $f^{-1}(W)$ is bounded in $X$, a contradiction as in that case, for $k > diam(f^{-1}(W))$,
$f^{-1}(U_k)$ can be covered by one set of diameter at most $2k$,
namely $f^{-1}(W)$. 

Apply \ref{CoarselyNto1Impliesnto1OnCoronas} to conclude the proof.
\end{proof}

\begin{Theorem}\label{GeneralCoarselyNTo1Thm}
If $n\ge 1$ and $f:X\to Y$ is a coarsely $n$-to-$1$ bornologous map of large scale spaces, then $asdim(Y)\leq asdim(X)+n-1$.
\end{Theorem}
\begin{proof}
\textbf{Case 1:} First consider the case of $X$ and $Y$ being metrizable.
It is known (see \cite{AV}) that $asdim(Y)$ is finite.
By \ref{AsdimVsDimOfHigsonCorona} it suffices to show
$\dim(\partial(\omega_Y))\leq \dim(\partial(\omega_X))+n-1$.
It is so (see
\ref{MainCoarseNTo1Thm}) as the induced map between boundaries of forms
is $n$-to-$1$ and Theorem 3.3.7 on page 196 in \cite{Engel} does say that dimension of the image of an $n$-to-$1$ map between compact spaces
is at most dimension of the range plus $(n-1)$.

\textbf{Case 2:} $Y$ is metrizable. Fix a metric $d_Y$ on $Y$ inducing its large scale structure.
 By induction construct a sequence of uniformly bounded covers
$\}\mathcal{V}_i\}_{i=1}^\infty$ of $X$ and a sequence of real numbers $r_i$ diverging to infinity
satisfying the following conditions:\\
1. Each element of $f^{-1}(\{B(y,r_i)\}_{y\in Y})$ is contained in a union of at most
$n$ elements of $\mathcal{V}_i$ for each $i\ge 1$,\\
2. $2\cdot r_i < r_{i+1}$ for each $i\ge 1$,\\
3. $st(\mathcal{V}_i,\mathcal{V}_i)$ refines $\mathcal{V}_{i+1}$ for each $i\ge 1$,\\
4. $ \mathcal{V}_i$ is of multiplicity at most $asdim(X)+1$ for each $i\ge 1$,\\
5. $f(\mathcal{V}_i)$ refines $\{B(y,r_{i+1})\}_{y\in Y}$ for each $i\ge 1$.

Define $d_X:X\times X\to [0,\infty]$ as follows:\\
a. $d_X(x,x)=0$ for each $x\in X$,\\
b. If $x\ne y\in X$ and there is $\ge 1$ such that both $x$ and $y$ belong to an element of $\mathcal{V}_k$, then $d_X(x,y)$ is the smallest such number $k$,\\
c. $d_X(x,y)=\infty$ if $x,y$ do not qualify in a) and b).

Consider the partition of $X$ into equivalence classes determined by $x\sim y$ if $d_X(x,y) < \infty$. There are at most $n$ such classes. Indeed,
if $S\subset X$ is a set containing $n+1$ elements pairwise non-equivalent, then $f(S)$ is contained in $B(y,r_i)$ for some $y\in Y$ and some $r_i$.
Therefore $S$ is contained in a union of at most
$n$ elements of $\mathcal{V}_i$, so two elements of $S$ belong to one one element of $\mathcal{V}_i$, a contradiction.

Notice that, for each equivalence class $X'$ of $X$, $f:X'\to f(X')$ is coarsely
$n$-to-$1$. By Case 1, $asdim(f(X'))\leq asdim(X)+n-1$.
Therefore $asdim(Y)\leq asdim(X)+n-1$.

\textbf{General Case:}
Suppose there is a uniformly bounded cover $\mathcal{U}_1$ of $Y$ which cannot be coarsened to a uniformly bounded cover of multiplicity at most
$asdim(X)+n$.
By induction construct a sequence of uniformly bounded covers
$\{\mathcal{V}_i\}_{i=1}^\infty$ of $X$ and a sequence of uniformly bounded covers
$\{\mathcal{U}_i\}_{i=1}^\infty$ of $Y$ satisfying the following conditions:\\
1. Each element of $f^{-1}(\mathcal{U}_i)$ is contained in a union of at most
$n$ elements of $\mathcal{V}_i$ for each $i\ge 1$,\\
2. $st(\mathcal{U}_i,\mathcal{U}_i)$ refines $\mathcal{U}_{i+1}$ for each $i\ge 1$,\\
3. $st(\mathcal{V}_i,\mathcal{V}_i)$ refines $\mathcal{V}_{i+1}$ for each $i\ge 1$,\\
4. $ \mathcal{V}_i$ is of multiplicity at most $asdim(X)+1$ for each $i\ge 1$,\\
5. $f(\mathcal{V}_i)$ refines $\mathcal{U}_{i+1}$ for each $i\ge 1$.

Notice that all covers of $X$ that refine $\mathcal{V}_i$ for some $i\ge 1$ form a large scale structure $\mathcal{LS}_X$ on $X$. Similarly, all covers of $Y$ that refine $\mathcal{U}_i$ for some $i\ge 1$ form a large scale structure $\mathcal{LS}_Y$ on $Y$. Also, $f:(X,\mathcal{LS}_X)\to (Y,\mathcal{LS}_Y)$ is coarsely $n$-to-$1$. Consider the following equivalence relation on $Y$: $y_1\sim y_2$ if there is $i\ge 1$ and an element $U$ of $\mathcal{U}_i$ containing both $y_1$ and $y_2$. 
By Case 2, each equivalence class of $Y$ is of asymptotic dimension at most $asdim(X)+n-1$. Consequently, for each equivalence class $Y'$ of $Y$
there is the smallest integer $k(Y')$ such that $\mathcal{U}_1 | Y'$
can be coarsened to a cover refining $\mathcal{U}_k$ and of multiplicity at most
$asdim(X)+n$. Therefore there is an infinite sequence $Y_i$ of equivalence classes of $Y$ so that $k(Y_i)$ is strictly increasing.
Put $Z:=\bigcup\limits_{i=1}^\infty Y_i$.
For each $i > 1$ redefine $\mathcal{U}_i$ by adding $\bigcup\limits_{j=1}^{i-1} Y_i$
to it. Similarly, each $i > 1$ redefine $\mathcal{V}_i$ by adding $\bigcup\limits_{j=1}^{i-1} f^{-1}(Y_i)$
to it. The new covers induce metrizable large scale structures on $Z$
and $f^{-1}(Z)$ such that $f:f^{-1}(Z)\to Z$ is coarsely $n$-to-$1$.
Moreover, $asdim(f^{-1}(Z))\leq asdim(X)$.
By Case 1, $asdim(Z)\leq asdim(X)+n-1$,
so there is a uniformly bounded cover $\mathcal{U}$ of $Z$ coarsening $\mathcal{U}_1|Z$
of multiplicity at most $asdim(X)+n$. That contradicts $k(Y_i)$ being strictly increasing.
\end{proof}

\begin{Remark}
Theorem \ref{GeneralCoarselyNTo1Thm} was proved by Austin-Virk in \cite{AV} for proper metric spaces $X$ and $Y$.
\end{Remark}

\begin{Theorem}\label{HomeoOfHigsonImpliesLSEquivalence}
A coarse bornologous function $f:X\to Y$ of metrizable large scale spaces is a large scale equivalence if and only if it induces a homeomorphism of Higson coronas.
\end{Theorem}
\begin{proof}
Suppose $f:X\to Y$ induces a homeomorphism of Higson coronas.\\
\textbf{Case 1}: $X\subset Y$ and $f$ is the inclusion. Given a metric $d$ on $Y$ we need to show existence of $r > 0$ such that $B(X,r)=Y$.
If such $r$ does not exist, then $Y_i:=Y\setminus B(X,i)$ is a sequence of unbounded sets of $Y$ and there is $\mathcal{P}\in \partial(\omega_Y)$ containing a sequence $y_i\in Y_i$. There is $\mathcal{Q}\in \partial(\omega_X)$, $\tilde f(\mathcal{Q})=\mathcal{P}$.
Since $X\in \mathcal{Q}$, $X\in \mathcal{P}$ leading to a contradiction:
$\omega_Y(X,\{y_i\}_{i=1}^\infty)=0$.\\
\textbf{General Case}:
By Case 1, $Z:=f(X)$ induces a large scale equivalence $i:Z\to Y$.
Apply \ref{MainCoarseNTo1Thm} to conclude $f:X\to Z$ is coarsely $1$-to-$1$. That means precisely that the inverse of a uniformly bounded cover of $Z$ is uniformly bounded in $X$ from which it follows that $f$ is a large scale equivalence.
\end{proof}

\begin{Corollary}
A coarse bornologous function $f:X\to Y$ of metrizable large scale spaces is a large scale embedding if and only if it induces a topological embedding of Higson coronas.
\end{Corollary}
\begin{proof}
$f:X\to f(X)$ induces a homeomorphism of Higson coronas. Apply \ref{HomeoOfHigsonImpliesLSEquivalence}.
\end{proof}

\section{Metrizability of form compactifications}

\begin{Theorem}
Suppose $(X,\omega)$ is a normal $T_1$ formed set. The following conditions are equivalent:\\
1. $X\cup\partial(\omega)$ is metrizable,\\
2. There exists a countable family $\mathcal{S}$ of subsets of $X$ such that for any subsets $C_1,\ldots,C_k$ of $X$ satisfying $\omega(C_1,\ldots,C_k)=0$ there are $D_i\in \mathcal{S}$, $i\leq k$, such that $C_i\subset D_i$ for each $i\leq k$ and $\omega(D_1,\ldots,D_k)=0$,\\
3. There exist subsets $C_n$, $n\ge 1$, of $X$ such that for any two subsets $D, E$ of $X$ satisfying $\omega(D,E)=0$ there is $C_k$ such that $D\subset C_k$ and $\omega(C_k,E)=0$.
\end{Theorem}
\begin{proof}
1)$\implies$2). Given a metric $d$ on $X\cup\partial(\omega)$, notice that complements $B_n$ of all balls $B(\partial(\omega),\frac{1}{n})$, $n\ge 1$, are $\omega$-bounded and we start $\mathcal{S}$ by enlisting all of them. As $(\partial(\omega),d)$ is compact metric, there is a countable family $\{D\}$ of subsets of $X$ such that sets $o(D)$ form a basis at points of $\partial(\omega)$.
Add all finite unions of elements of $\{D\}\cup \{B_n\}_{n=1}^\infty$ to $\mathcal{S}$.
If subsets $C_1,\ldots,C_k$ of $X$ satisfy $\omega(C_1,\ldots,C_k)=0$,
then the intersection of sets $cl(C_i)\cap \partial(\omega)$ is empty,
so there is a finite union $E_i$ of elements of $\{D\}$ such that
$C_i\setminus E_i$ is $\omega$-bounded and the intersection of sets $cl(E_i)\cap \partial(\omega)$ is empty. Each set $C_i\setminus E_i$
must be contained in some $B_n$, so all of them are contained in a single $B_m$. Define $D_i$ as $E_i\cup B_m$ and we are done.

3)$\implies$1). Notice that $\omega$-bounded elements of the family $\{C_n\}$ form a basis of $\omega$-bounded subsets of $X$ in the following sense: if $B$ is $\omega$-bounded, then there is $n$ such that $C_n$ is $\omega$-bounded and $B\subset C_n$. List those sets as $\{B_n\}$.

\textbf{Claim}: The family $\{int(cl(C_n)\setminus B_m)\}$ forms a basis of points of $\partial(\omega)$ in $X\cup \partial(\omega)$.\\
\textbf{Proof of Claim}: Suppose $\mathcal{Q}\in\partial(\omega)$ and $Q\in o(U)\subset cl(o(U))\subset o(W)$. Now, $\omega(U,X\setminus W)=0$,
so there exists $D\in\{C_n\}$ containing $U$ such that
$\omega(D,X\setminus W)=0$.
Choose $B_n$ containing $D\cap (X\setminus W)$ and notice $cl(D\setminus B_n)\cap cl(X\setminus W)=\emptyset$. Therefore
$Q\in int(cl(D\setminus B_n))\subset o(W)$. $\blacksquare$

Given a countable basis $\{U_n\}$ of points of $\partial(\omega)$ in $X\cup \partial(\omega)$ choose, for every pair $(U_n,U_m)$ such that $cl(U_n)\subset U_m$, a continuous function $f_{m,n}:X\cup \partial(\omega)\to [0,2^{-m-n}]$ such that $f(cl(U_n))\subset \{0\}$
and $f(X\cup \partial(X,\omega)\setminus U_m)\subset \{2^{-m-n}\}$.
Add all Dirac functions of points in $X$ that are $\omega$-bounded
and the resulting family $\{f_s\}_{s\in S}$ has continuous sum
$f:=\sum\limits_{s\in S}f_s$. Now, $\{f_s/f\}_{s\in S}$ is a partition of unity on
$X\cup \partial(\omega)$ whose carriers form a basis of $X\cup \partial(\omega)$. By \cite{JD3}, $X\cup \partial(\omega)$ is metrizable.

\end{proof}

\section{Forms on large scale spaces}

\begin{Proposition}
If $A$ is a subset of a large scale space $(X,\mathcal{L})$, then the form induced by $\mathcal{L}|A$ coincides with $\omega(\mathcal{L})|A$.
\end{Proposition}
\begin{proof}
Suppose $C_i\subset A$, $i\leq k$, and for each uniformly bounded cover $\mathcal{U}$ of $X$, the set $\bigcap\limits_{i=1}^k st(C_i,\mathcal{U})$ is bounded.
Suppose $\mathcal{W}$ is a uniformly bounded cover of $X$. To show
$\omega(\mathcal{L})(C_1,\ldots,C_k)=0$ it suffices to prove inclusion
$$\bigcap\limits_{i=1}^k st(C_i,\mathcal{W})\subset st(\bigcap\limits_{i=1}^k st(C_i,st(\mathcal{W},\mathcal{W})|A),\mathcal{W}).$$
If $x\in \bigcap\limits_{i=1}^k st(C_i,\mathcal{W})\setminus A$, then for each $i$ there is $W_i\in \mathcal{W}$ such that $C_i\cap W\ne \emptyset$ and $x\in C_i\cup W_i$. Choose $y_i\in C_i\cap W$ and notice
$y_j\in st(C_i,st(\mathcal{W},\mathcal{W})|A)$ for all $i,j\leq k$.
Therefore $x\in st(\bigcap\limits_{i=1}^k st(C_i,st(\mathcal{W},\mathcal{W})|A),\mathcal{W})$.
\end{proof}

\begin{Proposition}
Suppose $f,g:(X,\mathcal{L}_X)\to (Y,\mathcal{L}_Y)$ are coarse maps of large scale spaces. \\
a. If the induced form $\omega_Y:=\omega(\mathcal{L}_Y)$ is normal $T_1$ and $f$ is close to $g$, then
$\partial f=\partial g$.\\
b. If $(Y,\mathcal{L}_Y)$ is metrizable and $\partial f=\partial g$, then $f$ is close to $g$.
\end{Proposition}
\begin{proof} a. $\mathcal{U}:=\{f(x),g(x)\}_{x\in X}$ is a uniformly bounded family in $Y$.
Suppose $\mathcal{Q}\in \partial(\omega(\mathcal{L}_X)$
and $(\partial f)(\mathcal{Q})\ne (\partial g)(\mathcal{Q})$.
By \ref{ContExtOfFormContThm} there is $C\in \mathcal{Q}$ such that
$f(C)\notin (\partial g)(\mathcal{Q})$. 
Using \ref{ContExtOfFormContThm} again, we detect $D\in \mathcal{Q}$
so that $\omega_Y(f(C),g(D))=0$. Hence
$\omega_Y(st(f(C)),\mathcal{U}),g(D))=0$, a contradiction as
$g(C)\subset st(f(C)),\mathcal{U})$.

b. If $f$ is not close to $g$, then for each $k\ge 1$ there exists $x_k\in X$ such that $d(f(x_k),g(x_k)) > k$. Put $A:=\{x_k\}_{k=1}^\infty$ and notice it is unbounded. 
Notice $\omega_Y(f(A),g(A))=0$, so for any $\mathcal{Q}\in\partial(\omega(\mathcal{L}_X))$
containing $A$ one has $(\partial f)(\mathcal{Q})\ne(\partial g)(\mathcal{Q})$, a contradiction.

\end{proof}

\section{Asymptotic dimension 0}
It is well-known that a proper metric space $X$ is of asymptotic dimension $0$ if and only if its Higson corona is of dimension $0$. In this section we generalize this reasult to large scale structures induced by normal $T_1$ forms.
\begin{Definition}
A bornological large space $X$ is \textbf{coarsely totally disconnected} if $\mathcal{U}$-components of every uniformly bounded cover $\mathcal{U}$ of $X$ are bounded.
\end{Definition}

\begin{Lemma}\label{UComponentsLemma}
If $X$ is coarsely totally disconnected and $\mathcal{W}$ is the family of all $\mathcal{U}$-components of $X$ for some uniformly bounded cover $\mathcal{U}$, then $st(B,\mathcal{W})$ is bounded for all bounded sets $B$.
\end{Lemma}
\begin{proof}
If $st(B,\mathcal{W})$ is unbounded, then the $\mathcal{U}'$-component of $X$ containing $B$, $\mathcal{U}':=\mathcal{U}\cup\{B\}$, contains $st(B,\mathcal{W})$ and is unbounded, a contradiction.
\end{proof}

\begin{Theorem}
Suppose $\omega$ is a normal $T_1$ form on $X$.
If $\dim(X\cup\partial(\omega))=0$ and $LS(\omega)$ is coarsely totally disconnected, then
the asymptotic dimension of $LS(\omega)$ equals $0$.
\end{Theorem}
\begin{proof}
Suppose $\mathcal{U}$ is an element of $LS(\omega)$ such that $\mathcal{W}$, the collection of $\mathcal{U}$-components of $X$, is not uniformly bounded.
That means existence of an $\omega$-unbounded $C\subset X$ such that $X_0\cap cl(C)$
($X_0:=\partial(\omega)$) is a proper subset of $X_0\cap cl(st(C,\mathcal{W}))$.
Choose $\mathcal{Q}\in X_0$ belonging to $cl(st(C,\mathcal{W}))\setminus cl(C)$.
There is an open-closed subset $Y$ of $X\cup\partial(\omega)$ containing
$cl(C)$ and missing $\mathcal{Q}$. Let $Y_1:=X\cap Y$.

Choose a continuous function $\alpha:X\cup\partial(\omega)\to S^0$
such that $\alpha(Y_1)$ and $\alpha(X\setminus Y_1)$ are disjoint.
Consider $\mathcal{S}:=\{U\in \mathcal{U} | \alpha(U)=S^0\}$.
$B:=\bigcup \mathcal{S}$ is $\omega$-bounded as otherwise, for each
$U\in \mathcal{S}$, we can pick $x_U,y_U\in U$ satisfying $\alpha(x_U)=-1$
and $\alpha(y_U)=1$. In that case both $D:=\{x_U\}_{U\in \mathcal{S}}$
and $E:= \{y_U\}_{U\in \mathcal{S}}$ are $\omega$-unbounded with disjoint coronas
contradicting $E\subset st(D,\mathcal{U})$.

Consider $Y_1\setminus st(B,\mathcal{W})$ and $(X\setminus Y_1)\setminus st(B,\mathcal{W})$. There is no $\mathcal{U}$-chain joining these two sets, hence
the stars of the two sets with respect to $\mathcal{W}$ are disjoint,
contradicting $Q$ belonging to $cl(st(C,\mathcal{W}))\setminus cl(C)$.
\end{proof}

\begin{Example}
Given an infinite set $X$ consider $\omega$ defined as follows:
$\omega(V)=0$ if and only if at least one coordinate of $V$ is finite.
In that case $LS(\omega)$ consists of all covers $U$ of $X$ by finite sets
such that $st(F,U)$ is finite for each finite subset $F$ of $X$.
Notice $LS(\omega)$ is not coarsely totally disconnected. 
Indeed, choose an infinite sequence $\{x_n\}_{n=1}^\infty$ in $X$
and notice the trivial extension $\mathcal{U}$ of $\{x_n,x_{n+1}\}_{n=1}^\infty$
has infinite $\mathcal{U}$-component.
\end{Example}

\begin{Conjecture}
If If $\dim(X\cup\partial(\omega))=0$ and the asymptotic dimension of $LS(\omega)$ does not equal $0$, then
there is an unbounded subset $C$ of $X$ whose corona consists of exactly one point.
\end{Conjecture}

\section{Group actions on formed sets}
The goal of this section is to give sufficient conditions for finite group actions on $X$ to produce spaces of orbits $X/G$ of the same asymptotic dimension as $asdim(X)$.

\begin{Definition}
Suppose a group $G$ acts on a formed set $(X,\omega)$.
$\omega$ is \textbf{$G$-invariant} if
$$\omega(V)=\omega(g\cdot V)$$
for each $g\in G$ and each vector $V$ in $X$.

Naturally, $g\cdot (C_1,\ldots,C_k)$ is defined as $(g\cdot C_1,\ldots,g\cdot C_k)$.

The \textbf{space of orbits} $X/G$ consists of sets of the form $G\cdot x$. There is a natural projection $\pi:X\to X/G$ ($\pi(x):=G\cdot x$)
and $X/G$ is equipped with the form $\omega_G$ defined as follows:
$$\omega_G(C_1,\ldots,C_k)=\omega(\pi^{-1}(C_1),\ldots,\pi^{-1}(C_k))$$
for all vectors $(C_1,\ldots,C_k)$ in $X/G$.
\end{Definition}

\begin{Proposition}
Suppose $G$ is a finite group acting on $X$. If $\omega$ is a normal, $T_1$, and $G$-invariant form on $X$, then the induced form $\omega_G$ on $X/G$ is normal and $T_1$. Moreover, $\tilde \pi:X\cup \partial(\omega)\to
X/G\cup \partial(\omega_G)$ is a closed map.
\end{Proposition}
\begin{proof}
Given two different orbits $G\cdot x$ and $G\cdot y$, $\omega(G\cdot x,G\cdot y)=0$ (use additivity of $\omega$). Thus $\omega_G$ is $T_1$.

Supposae $\omega(\pi^{-1}(C_1),\ldots,\pi^{-1}(C_k))=0$
for a vector $(C_1,\ldots,C_k)$ in $X/G$. Choose
sets $D_i$ in $X$ so that $\omega(D_i,\pi^{-1}(C_i))=0$ for each $i\leq k$
and $\{D_i\}_{i=1}^k$ covers $X$. For each $g\in G$,
$\omega(g\cdot D_i,\pi^{-1}(C_i))=\omega(g\cdot D_i,g\cdot \pi^{-1}(C_i))=\omega(D_i,\pi^{-1}(C_i))=0$. Therefore, $\omega(G\cdot D_i,\pi^{-1}(C_i))=0$
and sets $E_i:=\pi(G\cdot D_i)$ form a cover of $X/G$
such that $\omega_G(E_i,C_i)=0$ for each $i\leq k$.

To show $\tilde \pi$ is closed we apply \ref{ClosedMapProp}
by observing $\pi(B)$ is $\omega_G$-bounded if $B$ is $\omega$-bounded.
Indeed, $\omega_G(\pi(B),X/G)=\omega(G\cdot B,X)=0$
as $G\cdot B$ is a finite union of $\omega$-bounded sets $g\cdot B$.
\end{proof}

\begin{Corollary}\label{DimensionOfCoronasUnderGroupActions}
Suppose $G$ is a finite group acting on $X$. If $\omega$ is a normal, $T_1$, and $G$-invariant form on $X$, then $\dim(X\cup \partial(\omega))=
\dim(X/G\cup \partial(\omega_G))$.
\end{Corollary}
\begin{proof} In order to apply Proposition 9.2.16 from \cite{Pears} 
we need to show that $\tilde \pi:X\cup \partial(\omega)\to 
X/G\cup \partial(\omega_G)$ is open and its fibers are finite. Here is the full statement of that proposition:\\
Let $X,Y$ be weakly paracompact, normal spaces. Let $f : X\to Y$ be a continuous, open surjection. If for every point $y\in Y$ the pre-image $f^{-1}(y)$ is finite, then $\dim(X) = \dim(Y)$.

Notice we can extend the action of $G$ over $X\cup \partial(\omega)$.
Namely, $g\cdot \mathcal{P}:=\{g\cdot C | C\in \mathcal{P}\}$.
Observe $g\cdot o(U)=o(g\cdot U)$, so $G$ acts by homeomorphisms
on $X\cup \partial(X,\omega)$.
It suffices to show that fibers of $\tilde \pi$ coincide with the orbits of $G$, then
$\tilde\pi$ is open as it is quotient (it is actually closed)
and the inverse of $\tilde\pi(W)$ is $G\cdot W$ for each open subset $W$
of $X\cup \partial(X,\omega)$.

Suppose $\tilde\pi(\mathcal{Q})=\tilde\pi(\mathcal{P})$ but $\mathcal{Q}\ne g\cdot \mathcal{P}$ for all $g\in G$. 
For each $g\in G$ there are open subsets $U_g$, $W_g$ of $X$
such that $\mathcal{Q}\in o(U_g)$, $g\cdot \mathcal{P}\in o(W_g)$
and closures of $o(U_g)$ and $o(W_g)$ are disjoint. In particular, $\omega(U_g,W_g)=0$.
Put $U=\bigcap\limits_{g\in G}U_g$ and $W= \bigcap\limits_{g\in G}g^{-1}\cdot W_g$. $\mathcal{Q}\in o(U)$, $\mathcal{P}\in o(W)$,
and $\omega(G\cdot U,G\cdot W)=0$, a contradiction
as $G\cdot U\in \tilde\pi(\mathcal{Q})$ and $G\cdot W\in\tilde\pi(\mathcal{P})$.
\end{proof}

\begin{Theorem}\label{GActingOnMetricSpaces}
Let $X$ be a metric space and let $G$ be a finite group acting isometrically on $X$. Then $X/G$ equipped with the Hausdorff metric has the same asymptotic dimension as $X$ if $asdim(X) < \infty$.
\end{Theorem}
\begin{proof}
Let $\omega$ be the basic large scale form on $X$ induced by its metric $d$. We will show that the basic large scale form $\lambda$ on $X/G$ induced by the Hausdorff metric $d_H$ equals the form $\omega_G$ on $X/G$ obtained from $\omega$ via the projection $\pi:X\to X/G$. Since both forms are normal and $T_1$, it suffices to show $\omega_G(C,D)=\lambda(C,D)$
for all $C,D\subset X/G$ (see \ref{NormalFormsOn2Vectors}).

Suppose $\omega_G(C,D)=\infty$. That means $B(G\cdot C,r)\cap B(G\cdot D,r)$ is unbounded for some $r > 0$. Pick points $c_i\in G\cdot C$ and $d_i\in G\cdot D$ such that $d(c_i,d_i) < 2r$ for each $i \ge 1$ and $\{c_i\}_{i=1}^\infty$ is unbounded. Notice $d_H(G\cdot c_i,G\cdot d_i) < 2r$ for each $i\ge 1$. Therefore $\lambda(G\cdot C,G\cdot D)=\infty$.

Conversely, if $\lambda(G\cdot C,G\cdot D)=\infty$, then there is $r > 0$
and sequences $c_i\in G\cdot C$ and $d_i\in G\cdot D$ such that $d_H(G\cdot c_i,G\cdot d_i) < 2r$ for each $i \ge 1$ and $\{G\cdot c_i\}_{i=1}^\infty$ is unbounded in $d_H$. We may pick the sequences $c_i\in G\cdot C$ so that $\{c_i\}_{i=1}^\infty$ is unbounded in $X$. For each $i$ we may choose
the closest point in $G\cdot d_i$ thus adjusting the sequence 
$d_i\in G\cdot D$ so that $d(c_i,d_i) < 2r $ for each $i \ge 1$. 
That means That means $B(G\cdot C,2r)\cap B(G\cdot D,2r)$ is unbounded and  $\omega_G(C,D)=\infty$. 

Notice $\pi:(X,d)\to (X/G,d_H)$ is coarsely $|G|$-to-$1$. Therefore
$asdim(X/G)$ is finite and equal to $\dim(\partial(\omega_G))=\dim(\omega_X)=asdim(X)$ by \ref{DimensionOfCoronasUnderGroupActions}.
\end{proof}

\begin{Remark}
In case of proper metric spaces $X$, Theorem \ref{GActingOnMetricSpaces}
was proved by Daniel Kasprowski \cite{Kasp}.
\end{Remark}

\section{Parallelism}
It is traditional in two-dimensional geometry to call two lines parallel if they are either equal or disjoint. A more natural definition of being parallel is to have the same coronas. However, on the hyperbolic plane only identical lines would be parallel under that definition.
\begin{Definition}
Suppose $\omega$ is a normal $T_1$ form on a set $X$.
$C\subset X$ is \textbf{parallel} to $D\subset X$ if their coronas are equal.
\end{Definition}

See \cite{KalHon} for other interesting equivalence relations in coarse geometry.

\begin{Proposition}
Suppose $\omega$ is a normal $T_1$ form on a set $X$.
$C\subset X$ is parallel to $D\subset X$ if and only if $E\perp C\iff E\perp D$ for all $E\subset X$, where $\perp$ is the orthogonality relation induced by $\omega$.
\end{Proposition}
\begin{proof}
$E\perp C$ means that the corona of $E$ is disjoint from the corona of $C$.
Therefore assume $E\perp C\iff E\perp D$ for all $E\subset X$ but
$\mathcal{P}$ is in the corona of $C$ but not in the corona of $D$.
As $D\notin \mathcal{P}$, there exists a vector $(C_1,\ldots,C_k)$ whose coordinates are in $\mathcal{P}$ but $\omega(D,C_1,\ldots,C_k)=0$.
By \ref{ContExtOfFormContThm} applied to $id_X$, there is $D'\perp D$ belonging to $\mathcal{P}$. Hence $D'\perp C$, a contradiction as both $D'$ and $C$ belong to $\mathcal{P}$.
\end{proof}

\section{Visual forms}\label{VisualForms}

In this section we consider proper geodesic spaces $(X,d)$ satisfying the following two conditions:\\
1. For every two points $x,y\in X$ there is a unique geodesic joining them,\\
2. For every point $p\in X$, $X$ is the union of all infinite geodesic rays emanating from $p$ (those spaces will be called \textbf{visual}).

Given a proper geodesic space $(X,d)$ with uniqueness of geodesics, for each two points $x,y\in X$ we denote the geodesics from $x$ to $y$ by $[x,y]$. The function $g(x,y,t)$, where $x,y\in X$ and $t\leq d(x,y)$ is the parametrization of $[x,y]$.

Given a geodesic ray $l$ emanating from $p\in X$, its parametrization is denoted by $g(p,l,t)$, $t\ge 0$.

\begin{Definition}
Given $C\subset X$ and $r > 0$ we define $N_p(C,r)$ as the closure of the set of all $\alpha(r)$, $\alpha$ ranging over all parametrizations $\alpha:[0,\infty)\to X$ of geodesic rays emanating from $p$ such that $\alpha[r,\infty)$ intersects $cl(C)$.

The form $\omega(N_p)$ induced by the function $N_p$ will be called the \textbf{visual form} on $(X,d)$ with respect to $p$.
\end{Definition}

\begin{Definition}
Given a subset $D$ of the $r$-sphere $S(p,r)$ centered at $p$, $Cone_p(D)$ is defined as the union of all $\alpha[r,\infty)$, $\alpha$ ranging over all parametrizations $\alpha:[0,\infty)\to X$ of geodesic rays emanating from $p$ such that $\alpha(r)\in D$.
\end{Definition}

\begin{Proposition}
If $D$ is a subset of the $r$-sphere $S(p,r)$ centered at $p$, 
then $N_p(Cone_p(D),r)=D$.
\end{Proposition}
\begin{proof}
Clearly, $N_p(Cone(D),r)\subset D$. Conversely, if $X\in D$, there is
a parametrization $\alpha:[0,\infty)\to X$ of a geodesic ray such that
$\alpha(r)=x$. Consequently, $D\subset N_p(Cone_p(D),r)$.
\end{proof}

\begin{Corollary}\label{VisualOrthogonalityProp}
If $D$ is a closed subset of the $r$-sphere $S(p,r)$ centered at $p$, 
and $D\subset S(p,r)\setminus N(C,r)$,
then $N_p(Cone_p(D),r)\cap N_p(C,r)=\emptyset$.
\end{Corollary}
\begin{proof}
Obvious, as $N_p(Cone_p(D),r)=D$.
\end{proof}

\begin{Corollary}\label{NormalityOfVisualForm}
$\omega(N_p)$ is a normal $T_1$ form that is compatible with the form $\omega_{ss}$ induced by the topology of $(X,d)$.
\end{Corollary}
\begin{proof}
To show normality of $\omega(N_p)$ suppose $\bigcap\limits_{i=1}^k N_p(C_i,r)=\emptyset$. Now, $\bigcup\limits_{i=1}^k (S(p,r)\setminus N_p(C_i,r))=S(p,r)$
and we can choose closed subsets $D_k$ of $S(p,r)$ such that
$D_i\subset S(p,r)\setminus N_p(C_i,r)$ for each $i\leq k$ and
$\bigcup\limits_{i=1}^k D_i=S(p,r)$.
By \ref{VisualOrthogonalityProp}, $\omega(N_p)(D_i,C_i)=0$ for each $i\leq k$ which proves $\omega(N_p)$ is normal.

$N_p(C,r)=\emptyset$ if and only if $cl(C)\subset B(p,r)$. Thus, $\omega(N_p)$-bounded sets are identical with pre-compact subsets of $X$.
Since $N_p(C,r)=N_p(cl(C),r)$, compatibility of forms follows from \ref{TopologicalCompatibilityCriterion}.
\end{proof}

\begin{Definition}
Fix $p\in X$. For each $n\ge 1$ let $\pi_n^{n+1}:S(p,n+1)\to S(p,n)$
be defined as $\pi_n^{n+1}(x)=g(p,x,n)$. Notice $\pi_n^{n+1}$ is continuous.
Given $C\subset X$ let $Inv_p(C)$ be the set of all sequences
$\{x_n\}_{n=1}^\infty$ such that $x_n\in S(p,n)$ and $\pi_n^{n+1}(x_{n+1})=x_n$ for each $n\ge 1$. If $C$ is unbounded, then $Inv_p(C)\ne\emptyset$
as $Inv_p(C)$ is identical with the inverse limit of $\{N_p(C),\pi_n^{n+1}\}$.
\end{Definition}

\begin{Proposition}\label{UniqueRay}
For each $\mathcal{P}\in \partial(\omega(N_p)$ there is a unique geodesic ray $l_{\mathcal{P}}$ at $p$ belonging to $\mathcal{P}$. Moreover, $C\in \mathcal{P}$ if and only if $N_p(l_{\mathcal{P}},C)=\infty$.
\end{Proposition}
\begin{proof}
Given $n\ge 1$, any family $C_1,\ldots,C_k\in \mathcal{P}$ induces non-empty intersection of $N_p(C_i,n)$, $i\leq k$. Hence all $N_p(C,n)$, $C\in \mathcal{P}$, have non-empty interesection denoted by $N_p(\mathcal{P},n)$. Observe
$\pi_n^{n+1}(N_p(\mathcal{P},n+1))\subset N_p(\mathcal{P},n)$ for each $n\ge 1$, so
there is a geodesic ray $l_{\mathcal{P}}$ at $p$ satisfying $g(p,l_{\mathcal{P}},n)\in N_p(\mathcal{P},n)$ for each $n\ge 1$.
Since $\{g(p,\{l_{P},n)\}= N_p(l_{\mathcal{P}},n)\in N_p(C,n)$ for each $n\ge 1$ and each $C\in \mathcal{P}$,
$l_{P}\in \mathcal{P}$. Any geodesic ray $l$ at $p$ belonging to $\mathcal{P}$
must satisfy $g(p,l,n)=g(p,l_{\mathcal{P}},n)$ for all $n\ge 1$, i.e. $l=l_{\mathcal{P}}$.
\end{proof}

\begin{Proposition}\label{BasisOfRays}
Given $\mathcal{P}\in \partial(\omega(N_p)$, the sets
$o(B(p,n+1)\cup Cone_p(B(l_{P},\epsilon)\cap S(p,n)))$ form a basis of open neighborhoods
of $\mathcal{P}$ in $\partial(\omega(N_p)+\omega_{ss})$.
\end{Proposition}
\begin{proof}
Suppose $U$ is open in $(X,d)$ and $\mathcal{P}\in o(U)$, i.e. $(X\setminus U)\notin \mathcal{P}$. Hence $N(l_{\mathcal{P}},n)\cap N_p(X\setminus U,n)=\emptyset$ for some $n\ge 1$ (see \ref{UniqueRay})
and there is $\epsilon > 0$ satisfying $cl(B(g(p,l_{P},\epsilon)))\cap N_p(X\setminus U,n)=\emptyset$.
Suppose $\mathcal{Q}\in o(B(p,n+1)\cup Cone_p(B(l_{P},\epsilon)\cap S(p,n)))$.
Therefore $Cone_p(B(l_{P},\epsilon)\cap S(p,n))\in \mathcal{Q}$,
so $N_p(\mathcal{Q},n)\subset B(l_{P},\epsilon)\cap S(p,n)$ (see the proof of \ref{VisualOrthogonalityProp}) and $N_p(\mathcal{Q},n)\cap (X\setminus U)=\emptyset$
resulting in $X\setminus U\notin \mathcal{Q}$ (see \ref{UniqueRay}),
i.e. $Q\in o(U)$.
\end{proof}

\subsection{CoG-spaces}

In this part we discuss a necessary and sufficient condition on $(X,d)$ for $\omega(N_p)=\omega(N_q)$ for all $p,q\in X$. That condition is satisfied by CAT(0)-spaces.

\begin{Definition}\label{ConvergenceOfGeoToARay}
A sequence of geodesics $[p,x_n]$ \textbf{converges} to a geodesic ray $l$ at $p$
if $\lim d(p,x_n)=\infty$ and
$\lim g(p,x_n,t)=g(p,l,t)$ for all $t\ge 0$.
\end{Definition}

\begin{Observation}
If one defines convergence of a sequence of geodesics $[p,x_n]$, $n\ge 1$, \textbf{converges} to a geodesics
$[p,x_0]$ analogously to \ref{ConvergenceOfGeoToARay}, then it simply amounts to $\lim x_n=x_0$.
\end{Observation}

In the next definition CoG is an initialism for "convergence of geodesics".
\begin{Definition}\label{CoGSpacesDef}
A proper geodesic visual space $(X,d)$ is a \textbf{CoG-space} if, 
whenever a sequence of geodesics $[p,x_n]$, $n\ge 1$, converges to a geodesic ray
$l_p$, then for every point $q\in X$ there is a geodesic ray $l_q$ at $q$ such that
 the sequence of geodesics $[q,x_n]$, $n\ge 1$, converges to $l_q$.
\end{Definition}

\begin{Proposition}\label{CAT0IsCoG}
Each CAT(0)-space $(X,d)$ is a CoG-space and $\partial(\omega(N_p))$ is the visual boundary at infinity of $(X,d)$.
\end{Proposition}
\begin{proof} CAT(0)-spaces have the property that for each geodesic ray $l_p$ at $p\in X$ and each $q\in X$ there is $K > 0$ and a geodesic ray $l_q$ at $q$ such that $d(g(p,l_p,t),g(q,l_q,t) < K$ for all $t\ge 0$
(see \cite{BH} , Proposition 8.2 on p.261).
Suppose a sequence of geodesics $[p,x_n]$, $n\ge 1$, converges to a geodesic ray
$l_p$ and let $l_q$ be a geodesic ray at $q$ such that for some constant $K > 0$ and all $r\ge 0$ one has $d(g(p,l_p,t),g(q,l_q,t) < K$ for all $t\ge 0$.
 Given $r > d(p,q)$ choose $m\ge 1$ so that $d(g(p,x_n,r),g(p,l_p,r)) < K$
for all $n\ge m$. Observe that $d(g(p,x_n,r),d(q,x_n,r))\leq d(p,q)$ (see the Claim below). Therefore, $d(q,x_n,r),g(q,l_q,r))
leq d(p,q)+d(p,q)+K$ for all $n\ge m$. That is sufficient to conclude $[q,x_n]$ converges to $l_q$. Indeed, given $k \ge 1$ and $s > 1$, put $r=s\cdot k\cdot (d(p,q)+d(p,q)+K)$ and find $m\ge 1$ such that $d(q,x_n,r),g(q,l_q,r))
leq d(p,q)+d(p,q)+K$ for all $n\ge m$. By applying the comparison to an euclidean triangle, we conclude $d(g(q,x_n,s),g(q,l_q,s))
\leq \frac{d(p,q)+d(p,q)+K}{k\cdot (d(p,q)+d(p,q)+K)}=1/k$ for all $n\ge m$.

\textbf{Claim}: Given a triangle $\Delta(PQR)$ on the euclidean plane
and $r \leq \min(PQ,PR)$, the distance $|Q'R'|$ between points $Q'\in PQ$ and $R'\in PR$ is at most $|QR|$ if $|QQ'|=r$ and $|RR'|=r$.\\
\textbf{Proof of Claim:} It is so if $Q'R'$ is parallel to $QR$.
Otherwise, without loss of generality, assume $Q'$ is higher than $R'$, i.e. the line $l$ parallel to $QR$ and passing through $R'$ intersects $QQ'$ at a point $Q''$. Consider the ray $l$ parallel to the ray $RP$ and emanating from $Q$. Look at the point $A$ on $l$ such that $|AQ|=r$.
Notice that $\angle(AQ'R') > \angle(R'AQ')$, so $|AR'| > |Q'R'$. 
\end{proof}

\begin{Theorem}\label{VisualFormIndependence}
Let $(X,d)$ be a proper geodesic visual space with uniqueness of geodesics.
The visual form does not depend on the base point $p$ if and only if $(X,d)$ is a CoG-space.
\end{Theorem}
\begin{proof}
Suppose $\omega(N_p)=\omega(N_q)$ for all $p,q\in X$ and $\{[p,x_n]\}_{n=1}^\infty$ is a sequence of geodesics converging to a geodesic ray $l$ at $p$.
If $\{[q,x_n]\}_{n=1}^\infty$does not converge to a geodesic ray at $q$,
then (by a double diagonal process) we can find two subsequences $C=\{c_n\}_{n=1}^\infty$ and $D=\{d_n\}_{n=1}^\infty$ of $\{x_n\}_{n=1}^\infty$ such that $\{[q,c_n]\}_{n=1}^\infty$ converges to a geodesic ray $l_C$ at $q$, $\{[q,d_n]\}_{n=1}^\infty$ converges to a geodesic ray $l_D$ at $q$, and $l_C\ne l_D$. Notice $\omega(N_q)(l_C,l_D)=0$ but $\omega(N_p)(l_C,l_D)=\infty$.

Suppose $(X,d)$ is a CoG-space.
Since both $\omega(N_p)$ and $\omega(N_q)$ are normal, it suffices to show $N_p(C,D)=0$ implies $N_q(C,D)=0$ for all $C,D\subset X$.

Suppose $N_p(C,r)\cap N_p(D,r)=\emptyset$ but 
$\omega(N_q)(C,D)=\infty$. Choose $\mathcal{P}\in \partial(\omega(N_q))$ containing both $C$ and $D$ and consider $l_{\mathcal{P}}$ (see \ref{UniqueRay}). There is a sequence $\{x_n\}_{n=1}^\infty$ such that
$[q,x_n]$ converges to $l_q:=l_{\mathcal{P}}$ and $x_k\in cl(C)$ for all $k$ odd,
$x_k\in cl(D)$ for all $k$ even. Let $l_p$ be the ray at $p$ such that
$[p,x_n]$ converges to $l_p$. Notice $l_p(r)\in N_p(C,r)\cap N_p(D,r)$, a contradiction.
\end{proof}

\begin{Proposition}
Suppose $(X,d_X)$ and $(Y,d_Y)$ are two CoG-spaces.
A function $f:X\to Y$ is visual form-continuous if and only if it has the following property: for any sequence of geodesics $\{[p,x_n]\}_{n=1}^\infty$ converging to a geodesic ray $l_p$ at $p\in X$, the sequence of geodesics $\{[q,f(x_n)]\}_{n=1}^\infty$ converges to some geodesic ray $l_q$ at $q=f(p)\in Y$.
\end{Proposition}
\begin{proof}
Suppose there is a sequence of geodesics $\{[p,x_n]\}_{n=1}^\infty$ converging to a geodesic ray $l_p$ at $p\in X$ such that the sequence of geodesics $\{[q,f(x_n)]\}_{n=1}^\infty$ does not converge to any geodesic ray at $q=f(p)\in Y$.
As in the proof of \ref{VisualFormIndependence} we can find two subsequences $C=\{c_n\}_{n=1}^\infty$ and $D=\{d_n\}_{n=1}^\infty$ of $\{x_n\}_{n=1}^\infty$ such that $\{[q,f(c_n)]\}_{n=1}^\infty$ converges to a geodesic ray $l_C$ at $q$, $\{[q,f(d_n)]\}_{n=1}^\infty$ converges to a geodesic ray $l_D$ at $q$, and $l_C\ne l_D$. Notice $\omega(N_p)(C,D)=\infty$ but $\omega(N_q)(f(C),f(D))=0$, so $f$ is not form-continuous.

Suppose $f$ is visual form-continuous and $\{[p,x_n]\}_{n=1}^\infty$  is a sequence of geodesics converging to a geodesic ray $l_p$ at $p\in X$ but the sequence of geodesics $\{[q,f(x_n)]\}_{n=1}^\infty$ does not converge to any geodesic ray at $q=f(p)\in Y$. As above, we can find two subsequences $C=\{c_n\}_{n=1}^\infty$ and $D=\{d_n\}_{n=1}^\infty$ of $\{x_n\}_{n=1}^\infty$ such that $\{[q,f(c_n)]\}_{n=1}^\infty$ converges to a geodesic ray $l_C$ at $q$, $\{[q,f(d_n)]\}_{n=1}^\infty$ converges to a geodesic ray $l_D$ at $q$, and $l_C\ne l_D$. Notice $\omega(N_p)(C,D)=\infty$ but $\omega(N_q)(f(C),f(D))=0$.
\end{proof}

\begin{Observation}
One can introduce the concept of a sequence of geodesic rays at $p$ to converge to a geodesic ray and show that the topology on $\partial(\omega(N_p)$ is induced by that convergence. However, it would be difficult to characterize
$f:X\to Y$ being visual form-continuous in terms of such convergence.
\end{Observation}

\begin{Problem}
Is there a CoG-space that is not a CAT(0)-space?
\end{Problem}

\end{document}